\documentclass{article}
\usepackage[a4paper, left=3cm, right=3cm, top=2cm]{geometry}
\usepackage{amsmath, amssymb}
\usepackage{hyperref,cleveref}
\usepackage{amsthm}
\usepackage{bm}
\usepackage{xcolor}
\usepackage{algorithm}
\usepackage{algpseudocode}
\usepackage{graphicx}
\usepackage{tikz}
\usetikzlibrary{arrows.meta,positioning}

\definecolor{workflowblue}{RGB}{48,108,190}
\definecolor{workflowgreen}{RGB}{64,140,77}
\definecolor{workfloworange}{RGB}{218,126,18}
\definecolor{workflowpurple}{RGB}{105,56,170}

\definecolor{workflowbluebg}{RGB}{244,248,254}
\definecolor{workflowgreenbg}{RGB}{245,250,244}
\definecolor{workfloworangebg}{RGB}{255,249,240}
\definecolor{workflowpurplebg}{RGB}{249,246,253}

\tikzset{
    workflowstage/.style={
        rounded corners=3mm,
        line width=0.9pt,
        minimum width=12.0cm,
        minimum height=2.2cm,
        inner sep=4mm
    },
    stagecircle/.style={
        circle,
        minimum size=7mm,
        inner sep=0pt,
        text=white,
        font=\bfseries
    },
    flow/.style={
        -{Latex[length=3mm,width=2mm]},
        thick,
        draw=black!50
    },
    outputnode/.style={
        rounded corners=1.5mm,
        line width=0.7pt,
        minimum width=4.6cm,
        minimum height=8mm,
        align=center,
        font=\bfseries
    },
    operatorbox/.style={
        rounded corners=2.5mm,
        line width=0.7pt,
        draw=black!35,
        fill=black!3,
        minimum width=4.2cm,
        minimum height=2.3cm,
        align=center,
        inner sep=4mm
    }
}
\usepackage{subcaption}

\newcommand{\R}{{\mathbb{R}}}

\newtheorem{theorem}{Theorem}
\newtheorem{proposition}{Proposition}
\newtheorem{lemma}{Lemma}
\newtheorem{corollary}{Corollary}
\newtheorem{definition}{Definition}

\newtheorem{remark}{Remark}
\usepackage[numbers,square]{natbib}
\usepackage{url}

\title{Randomized Symplectic Model Reduction: \\One-Pass Basis Construction and Sketched Hamiltonian Assembly}
\title{Randomized Symplectic Model Reduction: \\
One-Pass Basis Construction and Sketched Hamiltonian Assembly}

\author{%
R. Herkert\textsuperscript{1}%
\thanks{\texttt{robin.herkert@ians.uni-stuttgart.de}}
\quad
B. Haasdonk\textsuperscript{1}%
\thanks{\texttt{haasdonk@ians.uni-stuttgart.de}}
\quad
J. Fehr\textsuperscript{2}%
\thanks{\texttt{joerg.fehr@itm.uni-stuttgart.de}}
\\[1em]
\parbox{0.9\textwidth}{\centering\small
\textsuperscript{1}Institute of Applied Analysis and Numerical Simulation,\\
University of Stuttgart, Pfaffenwaldring 57,
70569 Stuttgart, Germany
\\[0.5em]
\textsuperscript{2}Institute of Engineering and Computational Mechanics,\\
University of Stuttgart, Pfaffenwaldring 9,
70569 Stuttgart, Germany
}
}

\date{}

\begin{document}
\maketitle

\begin{abstract}
Symplectic model order reduction of Hamiltonian systems yields reduced dynamics that conserve the reduced Hamiltonian along exact trajectories, avoiding the artificial energy drift that can arise from conventional projection methods. For large-scale spatial discretizations, however, basis construction and reduced-operator assembly can involve substantial offline costs. Randomized numerical linear algebra can reduce these costs through low-dimensional sketches while supporting streaming and parallel implementations. We develop randomized techniques for two parts of the offline stage while preserving the Hamiltonian structure of the reduced dynamics. First, we introduce
the one-pass sketched randomized cSVD (Ops-rcSVD), a one-pass randomized variant of the complex-SVD algorithm. Combining randomized range finding with an independent
left sketch, Ops-rcSVD
constructs an ortho-symplectic reduced basis from a single pass over the
snapshot ensemble. Second, for quadratic Hamiltonian systems, we introduce a
randomized reduced-Hamiltonian assembly based on sketching the reduced basis
and operator actions, yielding a reduced Hamiltonian system by construction. We derive bounds for the snapshot
projection error, uniform-in-parameter perturbation bounds for the sketched
reduced Hamiltonian, and an energy-norm estimate for the difference between
reduced and sketched reduced trajectories.
Experiments on large-scale systems show that Ops-rcSVD attains projection accuracy comparable to deterministic and two-pass randomized constructions while reducing basis-construction time. They further demonstrate accurate sketched reduced-Hamiltonian assembly when storing the complete snapshot ensemble is no longer feasible.
\end{abstract}
\section{Introduction}
Hamiltonian systems are widely used to model conservative dynamics, arising, for example, in mechanics and spatial discretizations of partial differential equations \cite{Meyer2017}. For large-scale discretizations, numerical simulations can become computationally expensive, particularly when repeated evaluations are required, as in parameter studies, optimization, or other multi-query settings. Moreover, applications such as real-time simulation and control often require rapid evaluations. Projection-based model order reduction (MOR) addresses these challenges by approximating the full-order model (FOM) dynamics in a low-dimensional subspace to reduce computational cost; see \cite{Benner2015} for a comprehensive overview. Classical projection-based reduction methods, however, do not in general preserve the underlying Hamiltonian structure. Consequently, invariants of the full-order dynamics, in particular energy conservation, may be lost in the ROM, leading to energy drift and qualitatively incorrect long-time behavior.

This motivates structure-preserving approaches, in particular symplectic MOR, where the reduced space is chosen so that the reduced dynamics remain Hamiltonian \cite{Peng2016,MaboudiAfkham2017,Buchfink2019}. For autonomous systems, this preserves the reduced Hamiltonian along exact reduced trajectories and, under additional assumptions, transfers qualitative stability properties of the full system to the ROM \cite{Meyer2017}.

A fundamental step in projection-based MOR is the construction of a reduced basis from representative snapshots. Such basis constructions can also be relevant for machine-learning-based MOR, for instance as a step in autoencoder-based approaches \cite{Fresca2022,Rettberg2025,brantner2023symplectic}. Snapshot-based symplectic MOR can be formulated through the proper symplectic decomposition, with practical basis constructions such as the cotangent lift and the complex SVD (cSVD), proposed in \cite{Peng2016}, as well as methods allowing non-orthonormal symplectic bases \cite{Buchfink2019,Xu2003}. Extensions include time-dependent reduced spaces and rank-adaptive strategies \cite{Pagliantini2021,HesthavenPagliantiniRipamonti2022,Schleuss23}, online-adaptive dictionary-based approximation \cite{Herkert2023dict}, and structure-preserving nonlinear-manifold methods such as quadratic manifolds and approximately symplectic autoencoders \cite{Sharma2023QM,Buchfink2023SMG}. Despite these advances, the offline phase can remain a major bottleneck: large snapshot ensembles must be generated, stored, and processed, low-rank bases computed, and reduced operators assembled.

Randomized numerical linear algebra reduces such costs by replacing expensive
deterministic matrix computations with approximate factorizations based on
randomized sketches \cite{Mahoney2010,Woodruff2014,Murray2023}. Core tools
include randomized low-rank approximation, range finding, and subspace
iteration \cite{Halko2011,Gu2015,Tropp2011Ana}, together with more recent
unified analyses of deterministic and randomized low-rank approximation
\cite{Demmel2023}. These methods can reduce computational and data storage costs while naturally supporting streaming environments and parallel implementations.

Randomization has also entered MOR through randomized SVD for reduced-basis construction \cite{Alla2019}, randomized low-rank approximation for nonlinear projection-based MOR \cite{Bach2019}, and randomized local MOR \cite{Buhr2018}. The computational limitations motivating randomized MOR also arise in realistic large-scale applications, where storing and factorizing the complete snapshot matrix can make conventional snapshot-based MOR prohibitively expensive. Examples include large-scale acoustic system identification from measured room-impulse-response data
\cite{PellingSarradj2026} and POD of computational-fluid-dynamics data sets obtained on grids with more than $10^9$ points \cite{MiroEtAl2026}.

In the Hamiltonian setting, randomized variants of the
complex-SVD construction (rcSVD) reduce the offline cost of symplectic basis
generation \cite{Herkert2023rand,herkert2025error}. However, conventional rcSVD still requires a second pass over the snapshot data to compute their coordinates in the basis obtained from the randomized sketch. If snapshots
are not retained, this involves regenerating and processing the training
trajectories, which is restrictive when snapshot generation is expensive or
the data are available only as a stream.

A systematic sketching framework for projection-based MOR was developed by
Balabanov and Nouy \cite{Balabanov2019,Balabanov2021}; see also
\cite{Balabanov2019a}. It constructs and evaluates reduced quantities through
low-dimensional sketches that approximately preserve norms and inner products
on relevant subspaces, thereby reducing offline computational and memory costs
and enabling randomized residual-based error estimation related to
\cite{Smetana19}. Directly sketched reduced systems, however, are not generally
guaranteed to retain Hamiltonian structure.

The aim of this work is therefore to combine randomized sketching with
symplectic MOR to reduce offline computational and data-access costs while
preserving the Hamiltonian structure of the reduced dynamics. We address two
major offline tasks: structure-preserving basis construction and reduced
Hamiltonian assembly. The main contributions of this work are:

\begin{enumerate}

\item
We develop the one-pass sketched randomized cSVD (Ops-rcSVD), a \emph{one-pass} randomized variant of the
complex-SVD-based symplectic basis construction, building on established
one-pass low-rank approximation and sketching techniques. Randomized range
finding is combined with an independent left sketch and sketched
least-squares recovery to produce an ortho-symplectic basis from a single
streaming pass over the snapshot ensemble. Thus, the complete snapshot matrix does not
need to be retained and the second data pass---or, in a streaming
implementation, second trajectory generation---required by conventional
rcSVD is avoided.

\item
For quadratic Hamiltonian systems, we
introduce a sketched reduced-Hamiltonian construction. Random sketches compress
the basis and Hamiltonian operator actions, while least-squares reconstruction
and explicit symmetrization yield a symmetric reduced Hessian and hence a
Hamiltonian ROM by construction. For parameter-separable Hamiltonians,
high-dimensional operations are confined to the offline stage, while online
assembly uses only precomputed reduced components.

\item
We analyze the errors introduced by both randomized stages. For Ops-rcSVD, we
bound the coefficient-recovery and snapshot projection errors, separating the
range-finding residual from the additional one-pass recovery error and relating
the approximation to the determinist cSVD projection error. For sketched
reduced-Hamiltonian assembly, subspace embedding of the basis and relevant
operator-residual spaces yields a uniform-in-parameter perturbation bound. For positive definite quadratic
Hamiltonians, we further derive an energy-norm trajectory-error bound with at
most linear growth in time. In both stages, the Hamiltonian structure is
preserved exactly.

\end{enumerate}

The methods target regimes in which the FOM dimension, number of snapshots,
trajectory-generation cost, or number of parameter-dependent operators makes
repeated processing of full-order data an offline bottleneck.

The paper is organized as follows. \Cref{Sec:SympMor} reviews symplectic MOR,
complex-SVD-based ortho-symplectic basis construction, and randomized
sketching. \Cref{sec:opsrcsvd} presents the Ops-rcSVD and its error, computational,
and memory analysis. \Cref{Sec:SketchedHam} introduces the sketched
reduced-Hamiltonian assembly and derives perturbation, and
trajectory-error guarantees under subspace-embedding conditions. Numerical experiments on a two-dimensional
Schr\"odinger problem and a lossless LC resonator network compare basis
construction methods and assess the large-scale randomized workflow when full
snapshot storage is infeasible. \Cref{Sec:Conclusion} summarizes the results
and future directions.

\section{Background on structure-preserving MOR and randomized subspace embeddings}
\label{Sec:SympMor}

We briefly review symplectic MOR for Hamiltonian systems, the cSVD construction of ortho-symplectic reduced bases,
and randomized subspace embeddings. For a more general introduction to MOR, we refer to
\cite{Benner2015,Benner2021,Benner2017}, and for symplectic MOR to
\cite{Peng2016,MaboudiAfkham2017,Buchfink2019}.

\subsection{Hamiltonian systems and symplectic MOR}
\label{Sec:HamSympMor}

Let $\mathcal P\subset\mathbb R^{n_\mu}$ and $I_t=[t_0,t_{\mathrm{end}}]$
denote a parameter domain and time interval. We consider the Hamiltonian $
\mathcal H(\cdot;\mu)\in C^1(\mathbb R^{2N};\mathbb R),$ the initial value $x_0(\mu)\in\mathbb R^{2N},$ and the parameter vector $\mu\in\mathcal P$.
For fixed $\mu$, a Hamiltonian system in canonical form is
\begin{equation}
\label{eq:FOM-Ham}
\frac{\mathrm d}{\mathrm dt}x(t;\mu)
=
J_{2N}\nabla_x\mathcal H(x(t;\mu);\mu),
\qquad
x(t_0;\mu)=x_0(\mu),
\end{equation}
where
\[
J_{2N}
:=
\begin{bmatrix}
0_N&I_N\\
-I_N&0_N
\end{bmatrix}
\]
denotes the canonical Poisson matrix. The state is commonly partitioned as
$x=[q^\top,p^\top]^\top$, with generalized position and momentum variables $q,p\in C^1(I_t;\mathbb R^N)$.

From the chain rule and the skew-symmetry of $J_{2N}$ it follows that the Hamiltonian is conserved along the exact trajectories.
Symplectic time integrators, such as the implicit midpoint rule, are commonly
used to retain the geometric structure after time discretization
\cite{Hairer2006GNI}.

Projection-based MOR approximates $
x(t;\mu)\approx Vx_r(t;\mu)$ with a basis matrix $
V\in\mathbb R^{2N\times2n}$ assuming $ n\ll N$.
For a test basis matrix $W\in\mathbb R^{2N\times2n}$, a Petrov--Galerkin projection
imposes
\[
W^\top
\left(
V\dot x_r
-
J_{2N}\nabla_x\mathcal H(Vx_r;\mu)
\right)
=
0.
\]
Without additional structure in the trial and test basis matrices $V$ and $W$, the reduced system is not guaranteed to preserve the Hamiltonian structure. Symplectic MOR therefore requires $V^\top J_{2N}V
=
J_{2n},$
and chooses the test basis through the symplectic inverse
\footnote{Throughout this work, $V^+$ denotes the symplectic inverse of
$V$, not the Moore--Penrose pseudoinverse.}
\begin{equation*}
V^+
:=
J_{2n}^{-1}V^\top J_{2N}
=
-\,J_{2n}V^\top J_{2N},
\qquad
W^\top:=V^+.
\end{equation*}
Then $V^+V=I_{2n}$ and, since $J_{2N}^2=-I_{2N}$, we have $V^+J_{2N}=J_{2n}V^\top.$
Defining the reduced Hamiltonian $\mathcal H_r(x_r;\mu)
:=
\mathcal H(Vx_r;\mu),$
the chain rule gives $\nabla_{x_r}\mathcal H_r(x_r;\mu) = V^\top\nabla_x\mathcal H(Vx_r;\mu),$ and hence the Petrov--Galerkin projection yields the reduced Hamiltonian system
\begin{equation}\label{eq:red-Ham}
\frac{\mathrm d}{\mathrm dt}x_r(t;\mu)
=
J_{2n}\nabla_{x_r}\mathcal H_r(x_r(t;\mu);\mu),
\qquad
x_r(t_0;\mu)=V^+x_0(\mu).
\end{equation}
Thus, the ROM is again a canonical Hamiltonian system.

\subsection{Symplectic basis generation via the complex SVD}
\label{Sec:cSVD}
We collect the snapshots in
\[
X_s 
:=
\begin{bmatrix}
x(t_1;\mu_1)&\cdots&x(t_{n_s};\mu_{n_s})
\end{bmatrix}
=
\begin{bmatrix}
Q_s\\P_s
\end{bmatrix}
\in\mathbb R^{2N\times n_s},
\]
where $(t_j,\mu_j)$, $j=1,\dots,n_s$, enumerates the sampled
parameter-time pairs and
$Q_s,P_s\in\mathbb R^{N\times n_s}$. The cSVD construction
\cite{Peng2016} introduces $Y_s := Q_s+\mathrm iP_s \in\mathbb C^{N\times n_s}.$ We consider a rank-$n$ truncated SVD
\[
Y_s\approx U\Sigma Z^*,
\qquad
U\in\mathbb C^{N\times n},
\qquad
U^*U=I_n.
\]
Writing $U=\Phi+\mathrm i\Psi$ and $\Phi,\Psi\in\mathbb R^{N\times n}$, 
orthonormality implies $\Phi^\top\Phi+\Psi^\top\Psi=I_n,$ and $\Phi^\top\Psi=\Psi^\top\Phi.$ The map
\begin{equation}
\label{eq:map}
\mathcal M(U)
:=
\begin{bmatrix}
\Phi&-\Psi\\
\Psi&\Phi
\end{bmatrix}
\in\mathbb R^{2N\times2n}
\end{equation}
therefore yields the cSVD basis matrix $V:=\mathcal M(U),
$ satisfying $V^\top V=I_{2n}$ and $
V^\top J_{2N}V=J_{2n}.$
The latter two conditions are called ortho-symplecticity. Hence, $V^+=V^\top$. Furthermore, an ortho-symplectic basis matrix admits the
block representation \eqref{eq:map}.

\subsection{Randomized subspace embeddings and sketched scalar products}
\label{Sec:RSE}

We recall the sketching concepts of
\cite{Balabanov2019,Balabanov2021}; see also \cite{Balabanov2019a}.
Throughout, $\mathbb K=\mathbb R$ or $\mathbb C$ and $\langle u,v\rangle=u^*v,$ $
\|u\|=\|u\|_2.$
A sketching map $\Theta:\mathbb K^m\rightarrow\mathbb K^k,$ where $k\ll m$, induces the sketched scalar product and sketched norm
\[
\langle u,v\rangle^\Theta
:=
\langle\Theta u,\Theta v\rangle,
\qquad
\|u\|^\Theta:=\|\Theta u\|_2.
\]
Weighted inner products can be treated similarly: if $M=L^*L$, then $\langle u,v\rangle_M
=
\langle Lu,Lv\rangle,$
so a Euclidean sketch can be applied after transformation by $L$; see
\cite{Woodruff2014}.

The main property used below is approximate norm preservation on a fixed
low-dimensional subspace.

\begin{definition}[$\varepsilon$-subspace embedding]
\label{def:subspace-embedding}
Let $Y\subset\mathbb K^m$ be a linear subspace. A linear map
$\Theta:\mathbb K^m\rightarrow\mathbb K^k$ is an
\emph{$\varepsilon$-subspace embedding for $Y$} if
\begin{equation}
\label{eq:subspace-embedding-norm}
(1-\varepsilon)\|y\|_2^2
\le
\|\Theta y\|_2^2
\le
(1+\varepsilon)\|y\|_2^2
\qquad
\forall\,y\in Y,
\end{equation}
for some $\varepsilon\in[0,1)$.
\end{definition}

The norm-preserving subspace-embedding property also implies approximate
preservation of inner products; see, e.g.,
\cite[Sec.~2.2]{PalittaSchweitzerSimoncini2025}.

\begin{lemma}[Approximate inner products from a subspace embedding]
\label{lem:ip_from_embedding}
Let $\Theta$ be an $\varepsilon$-subspace embedding for
$Y\subset\mathbb K^m$. Then
\begin{equation*}
\left|
\langle y_1,y_2\rangle
-
\langle\Theta y_1,\Theta y_2\rangle
\right|
\le
\varepsilon
\|y_1\|_2\|y_2\|_2
\qquad
\forall\,y_1,y_2\in Y.
\end{equation*}
\end{lemma}
Typical sketches include Gaussian matrices \cite{Tropp2011Ana}, structured randomized transforms
such as the subsampled randomized Hadamard transform (SRHT) and subsampled
randomized Fourier transform (SRFT) \cite{Tropp2011Ana}, and sparse embeddings such as
CountSketch \cite{Woodruff2014}.
The SRHT and SRFT take the respective forms
\[
\Theta
=
\sqrt{\frac{m}{k}}\,RHD,
\qquad
\Theta
=
\sqrt{\frac{m}{k}}\,RFD,
\]
where $R$ uniformly samples $k$ rows, $H$ is the normalized Walsh--Hadamard transform matrix, and $F$ is the unitary matrix of the discrete Fourier transform. For the SRHT, \(D\) contains independent random signs, whereas for the SRFT it contains independent random phases uniformly distributed on the complex unit circle. Both transforms admit efficient matrix-free implementations
with a computational cost of $\mathcal O(m\log m)$
operations per vector.

CountSketch assigns each input coordinate to one randomly selected output
coordinate together with a random sign. For \(i=1,\ldots,m\), we draw
\(h(i)\in\{1,\ldots,k\}\) and \(\sigma_i\in\{-1,+1\}\), with equal
probability for both signs, and set \(\Theta_{h(i),i}=\sigma_i\), while all
other entries in column \(i\) are zero.

Hence, for \(A\in\mathbb K^{m\times r}\), the sketch can be accumulated
row-wise as
\((\Theta A)_{h(i),:}\leftarrow(\Theta A)_{h(i),:}+\sigma_iA_{i,:}\).
Its application therefore requires only one signed accumulation per input
entry and costs \(\mathcal O(\operatorname{nnz}(A))\), where $\operatorname{nnz}(A)$ denotes the number of non-zero entries of $A$. Neither \(\Theta\)
nor a dense product with it needs to be formed, since the indices \(h(i)\)
and signs \(\sigma_i\) can be stored or regenerated from their random
parameters.

This structure is particularly useful when assembling reduced operators of the
form \(V^\top H V\), where \(V\in\mathbb R^{2N\times 2n}\) is a reduced basis matrix
and \(H\in\mathbb R^{2N\times 2N}\) is sparse. In this case, each row of
\(HV\) can be sketched immediately after it is computed, so that
\(\Theta_HHV\) is accumulated without forming the tall matrix \(HV\).
Likewise, \(\Theta_HV\) can be formed in \(\mathcal O(Nn)\) operations, up to
constant factors, independently of \(k_H\), whereas a dense embedding requires
\(\mathcal O(k_HNn)\) operations. Thus, CountSketch can be fused with the
sparse operator application while avoiding explicit storage of both
\(\Theta_H\) and the intermediate matrix \(HV\)
\cite{Woodruff2014}.

The framework of \cite{Balabanov2019,Balabanov2021} represents
high-dimensional reduced objects through low-dimensional sketches. For
example, for $V\in\mathbb K^{m\times r},$ and $V^\Theta:=\Theta V\in\mathbb K^{k\times r},$ one stores $V^\Theta$ together with sketches of relevant operator-applied
bases. Reduced matrices, scalar products, and residual norms can then be
assembled from objects with row dimension $k$ instead of $m$, although
the cost of forming the sketches depends on the full-order operators and
sketch structure.

In \Cref{sec:opsrcsvd}, we combine randomized range finding with a sketched
least-squares reconstruction to construct an ortho-symplectic basis without
revisiting the snapshot ensemble. In \Cref{Sec:SketchedHam}, the same
sketching principle is applied to reduced-Hamiltonian assembly, where
least-squares reconstruction followed by explicit symmetrization yields a
Hamiltonian reduced system.

\subsection{Randomized symplectic basis generation via rcSVD}
\label{sec:rcsvd}

The cSVD of \Cref{Sec:cSVD} constructs an ortho-symplectic basis from the
dominant left singular subspace of $Y_s\in\mathbb C^{N\times n_s}$.
For large problems (i.e., snapshot matrices with hundreds of millions to billions of entries),
assembling and factorizing $Y_s$ may require substantial runtime and memory. Randomized low-rank approximation instead identifies
an approximate dominant range from a low-dimensional random sample and
performs the remaining SVD in the compressed space.

We set $\ell=n+p\leq\min\{N,n_s\},$
where $n$ is the target complex basis dimension and $p\geq0$ an
oversampling parameter that improves the probability of capturing the
dominant left singular subspace \cite{Halko2011}. Next, we draw a right random sketching matrix $
\Omega\in\mathbb C^{n_s\times\ell}$, for example Gaussian or the adjoint of an SRFT
embedding. This matrix is distinct from the left sketching matrices $\Theta$ of
\Cref{Sec:RSE}, which act on the ambient state dimension.

The first rcSVD stage forms the random sketch
$Y:=Y_s\Omega\in\mathbb C^{N\times\ell}.$ Writing $Y_s=[y_1,\ldots,y_{n_s}]$ and denoting the $j$-th row of
$\Omega$ by $\omega_j^\top:=e_j^\top\Omega$, we obtain $Y
=
\sum_{j=1}^{n_s}y_j\omega_j^\top.$
Thus $Y$ can be accumulated sequentially as $Y\leftarrow Y+y_j\omega_j^\top,$ without storing the full snapshot matrix. If rows of $\Omega$ are generated
on demand, only the $N\times\ell$ random sketch $Y$, the random parameters, and
the current snapshot or block need be retained.

Provided $Y$ has full column rank, a thin QR factorization
\begin{equation*}
Y=QR,
\qquad
Q\in\mathbb C^{N\times\ell},
\qquad
Q^*Q=I_\ell,
\end{equation*}
produces an orthonormal basis for the sampled range.  With high probability,
$\operatorname{range}(Q)$ approximates the dominant left singular subspace
of $Y_s$.

Conventional rcSVD then performs a second snapshot pass to form $B
:=
Q^*Y_s
\in\mathbb C^{\ell\times n_s},$
so that $QQ^*Y_s
=
QB.$
The matrix $B$ contains the snapshot coordinates in
$\operatorname{range}(Q)$ and determines the dominant directions of the
projected ensemble. We compute an SVD
$B=\widehat U\,\Sigma\,\widehat V^*$ and assume
$\operatorname{rank}(B)\geq n$. Since $Q$ has
orthonormal columns, the leading left singular vectors of $QB$ are
\begin{equation*}
U_n^{\mathrm{rc}}
:=
Q\,\widehat U_{(:,1:n)}
\in\mathbb C^{N\times n},
\end{equation*}
with $\left(U_n^{\mathrm{rc}}\right)^*U_n^{\mathrm{rc}}=I_n.$
The matrix $U_n^{\mathrm{rc}}$ is mapped to a real, symplectic basis via
\[
V_{\mathrm{rc}}
:=
\mathcal M(U_n^{\mathrm{rc}})
\in\mathbb R^{2N\times2n}.
\]
Thus, rcSVD avoids the full SVD of $Y_s$ and does not need to retain the complete
snapshot ensemble. Its principal stored quantities are
$Q\in\mathbb C^{N\times\ell}$ and
$B\in\mathbb C^{\ell\times n_s}$, requiring storage complexity $\mathcal O(N\ell+\ell n_s)$ instead of $\mathcal O(Nn_s)$.

Its main limitation in a streaming setting is the second data pass:
$Q$ is obtained in the first pass, whereas
$B=Q^*Y_s$ requires revisiting the snapshots. If they are not stored,
the training trajectories must be regenerated and processed.

In the following two sections, we introduce randomized methods
to address two computational bottlenecks of symplectic MOR.
First, \Cref{sec:opsrcsvd} presents a one-pass basis construction
that avoids the second snapshot pass required by conventional rcSVD.
Second, \Cref{Sec:SketchedHam} introduces a sketched reduced-Hamiltonian
assembly to reduce the computational cost of constructing the
reduced system. dThe complete workflow and the interaction between the two
proposed methods, from one-pass basis construction to
sketched reduced-Hamiltonian assembly, are illustrated in
\Cref{fig:randomized_offline_workflow}.

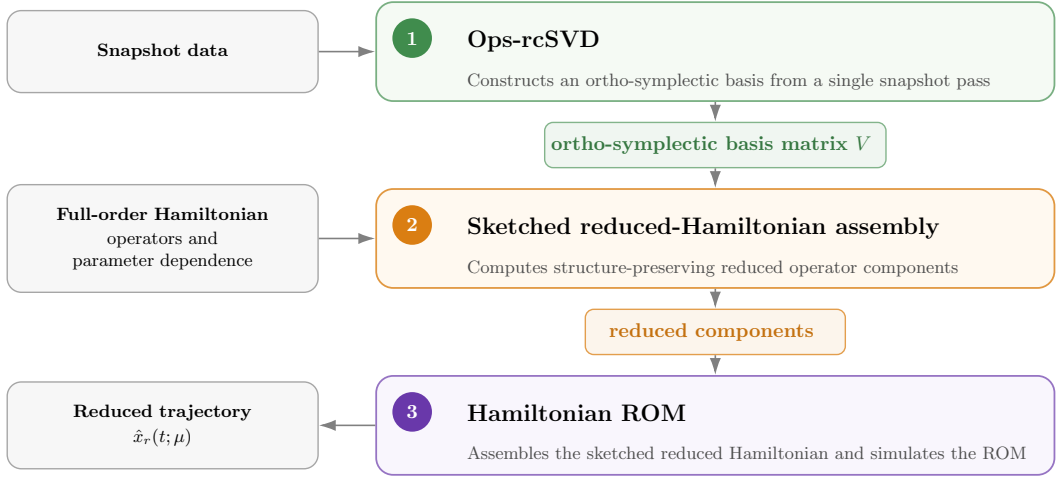
\begin{figure}[t]
\centering

\resizebox{0.96\textwidth}{!}{%
\begin{tikzpicture}[
    font=\small,
    >=Latex,
    sidebox/.style={
        operatorbox,
        minimum width=5.2cm,
        minimum height=1.45cm,
        text width=4.7cm,
        align=center
    }
]


\node[
    workflowstage,
    draw=workflowgreen!70,
    fill=workflowgreenbg,
    minimum height=1.75cm
] (ops) {};

\node[
    stagecircle,
    fill=workflowgreen,
    anchor=north west
] at ([xshift=4mm,yshift=-4mm]ops.north west) {1};

\node[
    font=\bfseries\large,
    anchor=north west
] at ([xshift=15mm,yshift=-4mm]ops.north west) {
    Ops-rcSVD
};

\node[
    anchor=north west,
    align=left,
    text=black!70,
    text width=11.0cm
] at ([xshift=15mm,yshift=-11.5mm]ops.north west) {
    Constructs an ortho-symplectic basis from a single snapshot pass
};


\node[
    sidebox,
    left=10mm of ops
] (snapshots) {
    \textbf{Snapshot data}
};

\draw[flow] (snapshots.east) -- (ops.west);


\node[
    outputnode,
    draw=workflowgreen!65,
    fill=workflowgreen!8,
    text=workflowgreen!85!black,
    below=3.5mm of ops
] (basisOutput) {
    ortho-symplectic basis matrix $V$
};

\draw[flow] (ops.south) -- (basisOutput.north);


\node[
    workflowstage,
    draw=workfloworange!80,
    fill=workfloworangebg,
    minimum height=1.75cm,
    below=3.5mm of basisOutput
] (ham) {};

\node[
    stagecircle,
    fill=workfloworange,
    anchor=north west
] at ([xshift=4mm,yshift=-4mm]ham.north west) {2};

\node[
    font=\bfseries\large,
    anchor=north west
] at ([xshift=15mm,yshift=-4mm]ham.north west) {
    Sketched reduced-Hamiltonian assembly
};

\node[
    anchor=north west,
    align=left,
    text=black!70,
    text width=11.0cm
] at ([xshift=15mm,yshift=-11.5mm]ham.north west) {
    Computes structure-preserving reduced operator components
};

\draw[flow] (basisOutput.south) -- (ham.north);


\node[
    sidebox,
    left=10mm of ham
] (operators) {
    \textbf{Full-order Hamiltonian}\\[0.5mm]
    operators and \\ 
    parameter dependence
};

\draw[flow] (operators.east) -- (ham.west);


\node[
    outputnode,
    draw=workfloworange!75,
    fill=workfloworange!8,
    text=workfloworange!90!black,
    below=3.5mm of ham
] (hamOutput) {
    reduced components
};

\draw[flow] (ham.south) -- (hamOutput.north);


\node[
    workflowstage,
    draw=workflowpurple!70,
    fill=workflowpurplebg,
    minimum height=1.75cm,
    below=3.5mm of hamOutput
] (rom) {};

\node[
    stagecircle,
    fill=workflowpurple,
    anchor=north west
] at ([xshift=4mm,yshift=-4mm]rom.north west) {3};

\node[
    font=\bfseries\large,
    anchor=north west
] at ([xshift=15mm,yshift=-4mm]rom.north west) {
    Hamiltonian ROM
};

\node[
    anchor=north west,
    align=left,
    text=black!70,
    text width=11.0cm
] at ([xshift=15mm,yshift=-11.5mm]rom.north west) {
    Assembles the sketched reduced Hamiltonian and simulates the ROM
};

\draw[flow] (hamOutput.south) -- (rom.north);


\node[
    sidebox,
    left=10mm of rom
] (trajectoryOutput) {
    \textbf{Reduced trajectory}\\[0.5mm]
    $\hat x_r(t;\mu)$
};

\draw[flow] (rom.west) -- (trajectoryOutput.east);

\end{tikzpicture}%
}
\caption{Randomized structure-preserving MOR workflow. Full-order simulations generate the
snapshot stream used by Ops-rcSVD to construct an ortho-symplectic basis. The basis and Hamiltonian operators then enter the sketched assembly, producing the components $\hat H_{r,q}$ from which
the Hamiltonian ROM is assembled online.}
\label{fig:randomized_offline_workflow}
\end{figure}

\section{One-pass sketched rcSVD}
\label{sec:opsrcsvd}

\subsection{One-pass construction and algorithm}
\label{sec:ops_rcsvd_construction}

Ops-rcSVD retains the randomized range-finding mechanism of the rcSVD while
eliminating its second snapshot pass. The difficulty is that the range basis
$Q$ is available only after the first pass, so the exact coefficient matrix
$B=Q^*Y_s$ cannot be accumulated while the snapshots are generated.

Following the standard single-pass randomized low-rank approximation paradigm
\cite{Halko2011}, we supplement the right range sketch with an independent
left sketching matrix $
\Theta\in\mathbb C^{k\times N},$ with $k\geq\ell.$

During the same pass in which $Y=Y_s\Omega$ 
is accumulated, we form
$Z:=\Theta Y_s\in\mathbb C^{k\times n_s}.$
For each snapshot $y_j$, we iteratively compute
\[
Y\leftarrow Y+y_j\omega_j^\top,
\qquad
Z_{(:,j)}\leftarrow\Theta y_j.
\] Thus, after one
streaming pass, $Y$ identifies the dominant range while $Z$ retains the
compressed information required for coefficient recovery.

As in rcSVD, we define
\[
Y=QR,\qquad Q^*Q=I_\ell,\qquad
P:=QQ^*,\qquad
B:=Q^*Y_s,\qquad
R_s:=(I-P)Y_s.
\]
Then, we obtain 
\begin{equation*}
Y_s=QB+R_s,
\qquad Q^*R_s=0,
\end{equation*}
and hence
\begin{equation*}
Z=(\Theta Q)B+\Theta R_s.
\end{equation*}
We therefore recover $B$ through the sketched least-squares problem
\begin{equation*}
B_{\mathrm{LS}}
:=
(\Theta Q)^\dagger Z
=
\arg\min_{X\in\mathbb C^{\ell\times n_s}}
\|\Theta QX-Z\|_F.
\end{equation*}
If $\Theta Q$ has full column rank, the minimizer is unique and
\begin{equation}
\label{eq:BLS_normal}
B_{\mathrm{LS}}
=
\bigl(Q^*\Theta^*\Theta Q\bigr)^{-1}
Q^*\Theta^*\Theta Y_s.
\end{equation}
In practice, the least-squares problem is solved by a thin QR factorization or an SVD of the small matrix $\Theta Q$, without explicitly forming a
pseudoinverse.

The two random maps have distinct roles: $\Omega$ identifies an approximate
dominant range of $Y_s$, whereas $\Theta$ preserves sufficient information
about the full snapshot ensemble to recover its coordinates after $Q$ is
known. Finally, we compute
$B_{\mathrm{LS}}=\widehat U\Sigma\widehat V^*$
and set $U_n:=Q\widehat U_{(:,1:n)}$. 
Then, we have $U_n\in\mathbb C^{N\times n}$ and $
U_n^*U_n=I_n$ and the symplectic basis matrix $V:=\mathcal M(U_n)\in\mathbb R^{2N\times2n}$
satisfies $V^\top V=I_{2n}$ and $V^\top J_{2N}V=J_{2n}.$ 
Thus, the one-pass reconstruction affects only the approximation of the
snapshot subspace; the final basis remains exactly ortho-symplectic in exact
arithmetic. 

\begin{algorithm}[t]
\caption{One-pass sketched rcSVD (Ops-rcSVD)}
\label{alg:onepass_rcsvd}
\begin{algorithmic}[1]
\Require Target rank $n$, oversampling parameter $p$, $\ell=n+p$, left-sketch
size $k\geq\ell$, right test matrix
$\Omega\in\mathbb C^{n_s\times\ell}$, and independent left sketch
$\Theta\in\mathbb C^{k\times N}$.
\Ensure Ortho-symplectic basis matrix $V\in\mathbb R^{2N\times2n}$.
\State $Y\gets0\in\mathbb C^{N\times\ell}$,
       $Z\gets0\in\mathbb C^{k\times n_s}$.
\For{$j=1,\ldots,n_s$}
    \State Read $y_j$ and the $j$-th row $\omega_j^\top$ of $\Omega$.
    \State $Y\gets Y+y_j\omega_j^\top$,
           $Z_{(:,j)}\gets\Theta y_j$.
\EndFor
\State Compute $Y=QR$ and $A\gets\Theta Q$.
\State Solve
\[
B_{\mathrm{LS}}
\gets
\arg\min_X\|AX-Z\|_F.
\]
\State Compute
$B_{\mathrm{LS}}=\widehat U\Sigma\widehat V^*$.
\State $U_n\gets Q\widehat U_{(:,1:n)}$,
       $V\gets\mathcal M(U_n)$.
\end{algorithmic}
\end{algorithm}

\begin{remark}[Batched processing and subsequent enrichment]
\label{rem:ops_batched_update}
Algorithm~\ref{alg:onepass_rcsvd} is written snapshot by snapshot for
simplicity. In the implementation, snapshots are processed in blocks. For a
block $Y_{\mathcal B}
=
[y_j]_{j\in\mathcal B}
\in\mathbb C^{N\times b}$
and the corresponding rows
$\Omega_{\mathcal B}\in\mathbb C^{b\times\ell}$, the two sketch updates are
performed as
\[
Y\leftarrow Y+Y_{\mathcal B}\Omega_{\mathcal B},
\qquad
Z_{(:,\mathcal B)}\leftarrow\Theta Y_{\mathcal B}.
\]
This is algebraically equivalent to the snapshot-wise formulation, but allows
the sketch updates to be carried out by matrix--matrix operations and limits
the required snapshot storage to one block. The block size controls the tradeoff between memory consumption
and the efficiency of matrix--matrix operations.
In our implementation, we use blocks of $b=16$ snapshots,
which keep the memory requirements moderate while enabling
blockwise sketch updates.

The sketching representation also permits the basis to be updated when new
snapshot data become available. We denote a new snapshot block by
$Y_{\mathrm{new}}\in\mathbb C^{N\times n_{\mathrm{new}}}$ and generate the corresponding new rows
$\Omega_{\mathrm{new}}\in\mathbb C^{n_{\mathrm{new}}\times\ell}$ of the
right test matrix. The accumulated sketches are extended as
\[
Y
\leftarrow
Y+Y_{\mathrm{new}}\Omega_{\mathrm{new}},
\qquad
Z
\leftarrow
\begin{bmatrix}
Z & \Theta Y_{\mathrm{new}}
\end{bmatrix}.
\]
A new QR factorization of the updated right sketch yields the enriched range
basis matrix $Q$. Since the left sketch $Z$ retains compressed information about
both the previous and newly added snapshots, the coefficient matrix can then
be reconstructed with respect to the updated range basis without revisiting the previously processed full-order snapshots. Updating
the SVD of $B_{\mathrm{LS}}$ subsequently provides the new reduced basis
and projector $P_n$.

Thus, subsequent enrichment requires retaining the accumulated sketches
$Y$ and $Z$. The previously computed
$B_{\mathrm{LS}}$ and $P_n$ alone are generally insufficient once
$\operatorname{range}(Q)$ changes.
\end{remark}
\subsection{Error analysis}
\label{sec:ops_rcsvd_error}

The one-pass method replaces the exact coefficients $B=Q^*Y_s$ by
$B_{\mathrm{LS}}$. Its error therefore separates into the randomized
range-finding residual and the additional left-sketch coefficient-recovery
error.

\begin{proposition}[Coefficient error bound for the left-sketch least-squares recovery]
\label{prop:coeff_error_B_BLS}
Let $Y_s\in\mathbb C^{N\times n_s}$ and
$Y=Y_s\Omega=QR$, with $Q^*Q=I_\ell$. Set $P:=QQ^*, B:=Q^*Y_s, R_s:=(I-P)Y_s,$
and $B_{\mathrm{LS}}
:=
(\Theta Q)^\dagger(\Theta Y_s).$
Assume that $\Theta\in\mathbb C^{k\times N}$ is an
$\varepsilon$-subspace embedding for $\operatorname{range}(Q)$,
with $0\leq\varepsilon<1$, and satisfies the approximate matrix
multiplication (AMM) bound
\begin{equation}\label{eqn:amm}
\left\|
Q^*(\Theta^*\Theta-I_N)R_s
\right\|_F
\le
\varepsilon_{\mathrm{AMM}}
\left(
\|Q\|_F
+
L\,\|Q\|_2
\right)
\|R_s\|_F
\end{equation}
for some $\varepsilon_{\mathrm{AMM}}>0$ and $L\geq0$.
Let $\gamma
:=
\varepsilon_{\mathrm{AMM}}
\left(
\|Q\|_F
+
L\,\|Q\|_2
\right)
=
\varepsilon_{\mathrm{AMM}}
(\sqrt{\ell} + L).$
Then
\begin{equation}
\label{eq:BLS_error_amm}
\|B_{\mathrm{LS}}-B\|_F
\le
\frac{\gamma}{1-\varepsilon}
\|(I-P)Y_s\|_F.
\end{equation}
\end{proposition}

\begin{proof}
Since $\Theta$ is an $\varepsilon$-subspace embedding for
$\operatorname{range}(Q)$, for every $z\in\mathbb C^\ell$, we have
\begin{equation}\label{eqn:embedding property}
(1-\varepsilon)\|z\|_2^2
\le
\|\Theta Qz\|_2^2
=
z^*Q^*\Theta^*\Theta Qz
\le
(1+\varepsilon)\|z\|_2^2
\end{equation}
due to \eqref{eq:subspace-embedding-norm} with $y = Qz$ and $\|Qz\|_2^2 = \|z\|_2^2$.
Hence $\lambda_{\min}(Q^*\Theta^*\Theta Q)
\ge
1-\varepsilon
>
0.$
Therefore $Q^*\Theta^*\Theta Q$ is invertible and
$\Theta Q$ has full column rank. In particular,
\eqref{eq:BLS_normal} is well defined.

Using the orthogonal decomposition $Y_s=QB+R_s, Q^*R_s=0,$
in \eqref{eq:BLS_normal} gives
\[
\begin{aligned}
B_{\mathrm{LS}} =
\bigl(Q^*\Theta^*\Theta Q\bigr)^{-1}
Q^*\Theta^*\Theta(QB+R_s)=
B+
\bigl(Q^*\Theta^*\Theta Q\bigr)^{-1}
Q^*\Theta^*\Theta R_s.
\end{aligned}
\]
Since $Q^*R_s=0$, the residual term can be written as $Q^*\Theta^*\Theta R_s
=
Q^*(\Theta^*\Theta-I_N)R_s.$ Consequently, it follows that
\[
B_{\mathrm{LS}}-B
=
\bigl(Q^*\Theta^*\Theta Q\bigr)^{-1}
Q^*(\Theta^*\Theta-I_N)R_s.
\]
The lower spectral bound supplied by the embedding property \eqref{eqn:embedding property} further implies
\[
\left\|
\bigl(Q^*\Theta^*\Theta Q\bigr)^{-1}
\right\|_2
=
\frac{1}{
\lambda_{\min}(Q^*\Theta^*\Theta Q)
}
\le
\frac{1}{1-\varepsilon}.
\]
Taking Frobenius norms and using
$\|AB\|_F\le\|A\|_2\|B\|_F$, we obtain
\[
\begin{aligned}
\|B_{\mathrm{LS}}-B\|_F
&\le
\left\|
\bigl(Q^*\Theta^*\Theta Q\bigr)^{-1}
\right\|_2
\left\|
Q^*(\Theta^*\Theta-I_N)R_s
\right\|_F
\\
&\le
\frac{1}{1-\varepsilon}
\left\|
Q^*(\Theta^*\Theta-I_N)R_s
\right\|_F.
\end{aligned}
\]
Finally, applying the AMM property \eqref{eqn:amm} and recalling that
$R_s=(I-P)Y_s$ yields \eqref{eq:BLS_error_amm}.
\end{proof}

In particular, if $(I-P)Y_s=0$, then $B_{\mathrm{LS}}=B$: the one-pass
recovery is exact whenever the randomized range contains the snapshot
ensemble.

\begin{remark}[High-probability embedding and AMM guarantees]
\label{rem:amm_bounds}
The coefficient-recovery bound of
Proposition~\ref{prop:coeff_error_B_BLS} relies on two properties of the
independent left sketch: a subspace-embedding property on
$\operatorname{range}(Q)$ and an AMM
property \eqref{eqn:amm} controlling its interaction with the residual $R_s$.

Since $\Theta$ is independent of the right test matrix $\Omega$, the
matrices $Q$ and $R_s$ may be regarded as fixed when analyzing the left
sketch. If the embedding property and the AMM bound \eqref{eqn:amm} hold with probabilities at least
$1-\delta_{\mathrm{emb}}$ and $1-\delta_{\mathrm{amm}}$, respectively,
then they hold simultaneously with probability at least
$1-\delta_{\mathrm{emb}}-\delta_{\mathrm{amm}}.$

For the complex SRFT used in Ops-rcSVD, Appendix~\ref{app:srft_amm} gives a
direct AMM estimate. With $\gamma_{\mathrm{SRFT}} = \frac{2(u+1)}{\sqrt{k}} \left( \sqrt{\ell}+L_\delta \right), L_\delta:=\sqrt{8\ln(N/\delta)}, 
0\leq u\leq\frac{\sqrt{k}}{1+L_\delta},$
Proposition~\ref{prop:srft_amm} yields, with probability at least
$1-e^{-u^2/4}-2\delta$,
\[
\|Q^*(\Theta^*\Theta-I_N)R_s\|_F
\le
\gamma_{\mathrm{SRFT}}\|R_s\|_F.
\]
Together with the corresponding subspace-embedding condition, these estimates
provide the assumptions of
Proposition~\ref{prop:coeff_error_B_BLS} with high probability.
\end{remark}

We next translate the coefficient estimate into a projection-error bound.

\begin{corollary}[Projection error bound for Ops-rcSVD]
\label{cor:proj_error_from_BLS_bound}
Adopt the notation of
Proposition~\ref{prop:coeff_error_B_BLS}, and let $1\le n\le\ell$. In particular, let $Y_s,\;Y,\;Q,\;P,\;B,\;R_s,\;$ and $B_{\mathrm{LS}}$
be defined as in Proposition~\ref{prop:coeff_error_B_BLS}.
Further, let
$B_{\mathrm{LS}}
=
\widehat U\Sigma\widehat V^*$
be an SVD of the recovered coefficient matrix, and define
$U_n:=Q\widehat U_{(:,1:n)},$
and 
$P_n:=U_nU_n^*.$
Furthermore, let $B_n$ be a best rank-$n$ approximation of $B$ in the spectral norm or Frobenius norm.
Then, for
$\|\cdot\|\in\{\|\cdot\|_2,\|\cdot\|_F\}$,
we obtain the bound
\begin{equation}
\label{eq:proj_err_basic}
\|(I-P_n)Y_s\|
\le
\|(I-P)Y_s\|
+
\|B-B_n\|
+
2\|B-B_{\mathrm{LS}}\|.
\end{equation}
Moreover, under the embedding and AMM assumptions of
Proposition~\ref{prop:coeff_error_B_BLS}, it holds that
\begin{equation}
\label{eq:proj_err_final_explicit}
\|(I-P_n)Y_s\|_F
\le
\|Y_s-(Y_s)_n\|_F
+
\left(
1+\frac{2\gamma}{1-\varepsilon}
\right)
\|(I-P)Y_s\|_F.
\end{equation}
\end{corollary}
\begin{proof}
Since $\operatorname{range}(P_n)
\subseteq
\operatorname{range}(P),$
we have $P_nP=P_n$. Therefore, we can decompose
\[
I-P_n
=
(I-P)+(I-P_n)P,
\]
and hence obtain
\[
(I-P_n)Y_s
=
(I-P)Y_s+(I-P_n)PY_s.
\]
Taking either the spectral or Frobenius norm and applying the triangle
inequality gives
\begin{equation}
\label{eq:proj_split_triangle}
\|(I-P_n)Y_s\|
\le
\|(I-P)Y_s\|
+
\|(I-P_n)PY_s\|.
\end{equation}
Recall that
$PY_s=QB, P_n=Q\Pi_nQ^*, \Pi_n
:=
\widehat U_{(:,1:n)}
\widehat U_{(:,1:n)}^*.$
Using $Q^*Q=I_\ell$, we obtain
\begin{equation}\label{eq:proj_error_equ}
(I-P_n)PY_s
=
(I-Q\Pi_nQ^*)QB
=
Q(I-\Pi_n)B.
\end{equation}
Since multiplication from the left by the column-orthonormal matrix $Q$
preserves both the spectral and Frobenius norms, we have
\[
\|(I-P_n)PY_s\|
=
\|(I-\Pi_n)B\|.
\]
We now compare $B$ with the recovered coefficient matrix
$B_{\mathrm{LS}}$. By the triangle inequality and
$\|I-\Pi_n\|_2\le1$, we get
\[
\begin{aligned}
\|(I-\Pi_n)B\|
&\le
\|(I-\Pi_n)B_{\mathrm{LS}}\|
+
\|(I-\Pi_n)(B-B_{\mathrm{LS}})\|
\\
&\le
\|(I-\Pi_n)B_{\mathrm{LS}}\|
+
\|B-B_{\mathrm{LS}}\|.
\end{aligned}
\]
Since $\Pi_n$ projects onto the leading $n$ left singular vectors of
$B_{\mathrm{LS}}$, we have
\[
\|(I-\Pi_n)B_{\mathrm{LS}}\|
=
\|B_{\mathrm{LS}}-(B_{\mathrm{LS}})_n\|.
\]
Since $B_n$ has rank at most $n$, it is a valid candidate in the
best rank-$n$ approximation problem for $B_{\mathrm{LS}}$. Thus, it holds that
\[
\begin{aligned}
\|B_{\mathrm{LS}}-(B_{\mathrm{LS}})_n\|
\le
\|B_{\mathrm{LS}}-B_n\|
\le
\|B-B_n\|
+
\|B-B_{\mathrm{LS}}\|.
\end{aligned}
\]
Combining the last three estimates yields
\[
\|(I-\Pi_n)B\|
\le
\|B-B_n\|
+
2\|B-B_{\mathrm{LS}}\|.
\]
Substitution this estimate into \eqref{eq:proj_split_triangle} while using \eqref{eq:proj_error_equ} proves
\eqref{eq:proj_err_basic}.

It remains to specialize the estimate to the Frobenius norm. Since
$PY_s=QB$ and $Q$ has orthonormal columns, multiplication by $Q$
preserves the singular values of $B$. Consequently, we have
\[
\|B-B_n\|_F
=
\|PY_s-(PY_s)_n\|_F.
\]
Now $P(Y_s)_n$ has rank at most $n$, so it is an admissible competitor
for the best rank-$n$ approximation of $PY_s$. Therefore, it follows that
\[
\begin{aligned}
\|PY_s-(PY_s)_n\|_F
\le
\|PY_s-P(Y_s)_n\|_F
=
\|P\bigl(Y_s-(Y_s)_n\bigr)\|_F
\le
\|Y_s-(Y_s)_n\|_F,
\end{aligned}
\]
where the last inequality uses that the orthogonal projector $P$ is
non-expansive in the Frobenius norm. Hence, we arrive at
\[
\|B-B_n\|_F
\le
\|Y_s-(Y_s)_n\|_F.
\]
Finally, Proposition~\ref{prop:coeff_error_B_BLS} gives
\[
\|B-B_{\mathrm{LS}}\|_F
\le
\frac{\gamma}{1-\varepsilon}
\|(I-P)Y_s\|_F.
\]
Applying these two estimates to \eqref{eq:proj_err_basic} yields \eqref{eq:proj_err_final_explicit}.
\end{proof}
\begin{remark}[Interpretation as symplectic projection error]
\label{rem:complex_real_projection_error}
For
\[
X_s=
\begin{bmatrix}
Q_s\\P_s
\end{bmatrix},
\qquad
Y_s=Q_s+\mathrm iP_s,
\qquad
V=\mathcal M(U_n),
\]
it holds for the orthogonal projection error that
\begin{equation}
\label{eq:real_complex_proj_identity}
\|(I_{2N}-VV^\top)X_s\|_F
=
\|(I_N-P_n)Y_s\|_F.
\end{equation}
If $V_{\mathrm c}=\mathcal M(U_n^{\mathrm c})$ is the deterministic cSVD
basis matrix, then
\[
\|(I_{2N}-V_{\mathrm c}V_{\mathrm c}^\top)X_s\|_F
=
\|Y_s-(Y_s)_n\|_F.
\]
Hence \eqref{eq:proj_err_final_explicit} becomes
\[
\|(I_{2N}-VV^\top)X_s\|_F
\le
\|(I_{2N}-V_{\mathrm c}V_{\mathrm c}^\top)X_s\|_F
+
\left(
1+\frac{2\gamma}{1-\varepsilon}
\right)
\|(I-P)Y_s\|_F.
\]
Thus, the deterministic cSVD error is augmented by the randomized range error
and the additional one-pass recovery contribution.
\end{remark}

\paragraph{Contribution of randomized range finding.}
For a Gaussian right test matrix $\Omega$ and
$p=\ell-n\geq2$, the classical estimate \cite[Theorem 10.5]{Tropp2011Ana}
\[
\mathbb E_\Omega\|(I-P)Y_s\|_F
\le
\left(1+\frac{n}{p-1}\right)^{1/2}
\left(\sum_{j>n}\sigma_j(Y_s)^2\right)^{1/2}
\]
shows that the expected range-finding error is controlled by the optimal
rank-$n$ singular-value tail. Corollary~\ref{cor:proj_error_from_BLS_bound} quantifies the additional one-pass contribution under the embedding and AMM assumptions imposed on the independent left sketch. Analogous estimates hold for structured right test
matrices such as the SRHT and SRFT.

\subsection{Computational cost, data access, and memory}
\label{sec:ops_rcsvd_cost}
Throughout the complexity analysis, we consider the
low-rank regime
\[
    \ell=n+p\ll\min\{N,n_s\},
    \qquad \ell\leq k\ll N. 
\]
No particular ordering between $N$ and $n_s$ is assumed. \footnote{If only one of the dimensions is large, the sketching
strategy can be adapted accordingly. For large $N$ but moderate
$n_s$, the right sketch may be omitted by constructing an
approximate basis directly during snapshot streaming, e.g.,
using an incremental SVD \cite{Brand2006,Fareed2018}.
The left sketch can still be used for coefficient recovery,
while orthonormalization of the final complex basis ensures
ortho-symplecticity. Conversely, for moderate $N$ but
large $n_s$, the left sketch may be omitted, recovering
the standard two-pass rcSVD of \cite{Herkert2023rand}.} 

Random test matrices and sketching maps are assumed not to be
stored explicitly. Their actions are implemented implicitly,
or their entries are regenerated reproducibly from the
corresponding random seeds or defining parameters.
Thus, their full matrix storage is excluded from the memory
estimates. In particular, explicitly forming and storing an
SRFT as a dense matrix would eliminate much of its
computational advantage, since both its memory footprint
and application cost would then scale as for a dense embedding.

Both rcSVD and Ops-rcSVD begin by forming the right range sample
$Y=Y_s\Omega$. In a snapshot-wise implementation, accumulating
$Y=\sum_{j=1}^{n_s}y_j\omega_j^\top$ 
requires $\mathcal O(N\ell)$ operations per snapshot and hence
$\mathcal O(N\ell n_s)$ operations in total. Then, $\mathcal O(N\ell^2)$ operations are required for the thin QR factorization. Ops-rcSVD also forms the left
snapshot sketch $Z=\Theta Y_s$ and the sketched range basis matrix $\Theta Q$.  With a structured transform such as the SRFT, the corresponding sketch
applications can instead be carried out in $\mathcal O\!\left((n_s+\ell)N\log N\right)$
operations. Once the sketches have been formed, the remaining least-squares
recovery and compressed SVD depend only on $k$, $\ell$, and $n_s$.
For $k\geq\ell$ and $\ell\leq n_s$, their leading cost is $
\mathcal O(k\ell n_s+\ell^2n_s).$ The principal snapshot-access and storage requirements are summarized in
Table~\ref{tab:basis-memory-comparison}.
\begin{table}[t]
    \centering
    \caption{Leading snapshot-access, arithmetic, and storage requirements under the assumptions
$\ell=n+p\ll\min\{N,n_s\}$ and $\ell\leq k\ll N$.
    Random maps are represented implicitly and lower-order terms are omitted.
    The Ops-rcSVD arithmetic cost assumes an SRFT left sketch.}
    \label{tab:basis-memory-comparison}
    \begin{tabular}{lccc}
        \hline
        & cSVD & rcSVD & Ops-rcSVD \\
        \hline
        Snapshot access
        & complete ensemble
        & two passes
        & one pass
        \\
        Computational cost
        & $\mathcal O(Nn_s\min\{N,n_s\})$
        & $\mathcal O(N\ell n_s)$
        & $\mathcal O(N\ell n_s+Nn_s\log N+k\ell n_s)$
        \\
        Storage
        & $\mathcal O(Nn_s)$
        & $\mathcal O(N\ell+\ell n_s+Nb)$
        & $\mathcal O(N\ell+kn_s+Nb)$
        \\
        \hline
    \end{tabular}
\end{table}
The principal distinction between streaming rcSVD and Ops-rcSVD is therefore one of data access and storage. The standard cSVD retains the complete snapshot matrix.
Streaming rcSVD, a variant of the rcSVD that forms the range sample $Y=Y_s\Omega$ sequentially by processing the snapshot ensemble blockwise and revisits the snapshots to form $B=Q^*Y_s$, avoids this storage cost but requires a second data pass. Ops-rcSVD removes this second access by storing the compressed snapshot sketch
$Z=\Theta Y_s\in\mathbb C^{k\times n_s},$
together with the reduced coefficient data. It therefore requires more memory
than streaming rcSVD, but its snapshot-dependent storage scales with the reduced
dimensions $k,\ell\ll N$ rather than with the full spatial dimension $N$.

A streaming block of $b$ complex double-precision snapshots occupies $16Nb$ bytes, independently of the total number of snapshots $n_s$. Although rcSVD and Ops-rcSVD have comparable arithmetic complexity, a second pass over the snapshot  ensemble can still be costly. If the snapshots are not retained in RAM, they must either be regenerated by repeating the full-order simulations or read again from external storage. For large snapshot ensembles, repeated transfers from external storage can contribute substantially to the overall offline runtime. These data-generation and data-transfer costs are not reflected by the floating-point operation count. Ops-rcSVD avoids this second access by retaining the compressed left sketch during the first pass. Its main advantage is therefore the reduction of snapshot data access rather than a reduction in floating-point work. The numerical experiments in \Cref{Sec:NumExp} quantify this tradeoff in terms of runtime and approximation accuracy.

\section{Sketched reduced-Hamiltonian assembly}
\label{Sec:SketchedHam}
We fix an ortho-symplectic reduced basis matrix
$V\in\mathbb{R}^{2N\times 2n}$. We consider quadratic Hamiltonians
of the form
\[
\mathcal{H}(x;\mu)
=
\tfrac12 x^\top H(\mu)x,
\qquad
H(\mu)=H(\mu)^\top,
\]
for which the reduced dynamics are
\begin{equation}\label{eq:red-lin-Ham_sys}
\dot x_r(t;\mu)
=
J_{2n}H_r(\mu)x_r(t;\mu),
\qquad
H_r(\mu):=V^\top H(\mu)V,
\qquad
x_r(t_0;\mu)=V^+x_0(\mu).
\end{equation}
If the Hamiltonian admits the parameter-separable decomposition
\begin{equation}
\label{eq:param_sep_H_motivation}
H(\mu)
=
\sum_{q=1}^{Q}
\vartheta_q(\mu)H_q,
\qquad
H_q=H_q^\top,
\end{equation}
then the reduced Hamiltonian inherits the corresponding representation
\[
H_r(\mu)
=
\sum_{q=1}^{Q}
\vartheta_q(\mu)H_{r,q},
\qquad
H_{r,q}:=V^\top H_qV.
\]
The matrices $H_{r,q}$ are parameter-independent and can therefore be
assembled offline. For PDE-scale problems, however, their formation may still
be costly because it requires high-dimensional operator applications and
reduced inner products. We replace this exact offline assembly by a sketched construction that reduces
the high-dimensional cost while retaining Hamiltonian structure.

\subsection{Sketched Hamiltonian construction and offline/online decomposition}
\label{sec:sketched_ham_construction}

We use a randomized subspace embedding
$\Theta_H\in\mathbb{R}^{k_H\times 2N}$ with $k_H\ll 2N$. For the probabilistic analysis below, we
assume that $\Theta_H$ is drawn independently of any randomness used in the
construction of the reduced basis matrix $V$. We begin with the sketched quadratic form
\begin{equation}
\label{eq:raw_sketch_ham}
\mathcal{H}^{\Theta_H}(x;\mu)
:=
\tfrac12(\Theta_Hx)^\top\Theta_HH(\mu)x
=
\tfrac12x^\top
\Theta_H^\top\Theta_HH(\mu)x.
\end{equation}
Since the skew-symmetric part of a matrix does not contribute to a quadratic
form, this expression depends only on
\[
H^{\Theta_H}(\mu)
:=
\operatorname{sym}
\bigl(\Theta_H^\top\Theta_HH(\mu)\bigr),
\qquad
\operatorname{sym}(A):=\tfrac12(A+A^\top).
\]
The full matrix $H^{\Theta_H}(\mu)$ serves only as an interpretation of the
sketched quadratic form and is never assembled explicitly. Restricting \eqref{eq:raw_sketch_ham} to the reduced space gives
\begin{equation*}
\mathcal{H}^{\Theta_H}_r(x_r;\mu)
=
\tfrac12x_r^\top M(\mu)x_r,
\end{equation*}
with
\begin{equation*}
M(\mu)
:=
(\Theta_HV)^\top\Theta_HH(\mu)V.
\end{equation*}
Note that we do not use the symmetric double-sided construction $(\Theta_H^\top\Theta_HV)^\top
H_q
(\Theta_H^\top\Theta_HV),$
since it introduces an additional sketch-induced approximation without reducing
the dominant reduced-product cost. In both the one- and double-sided
constructions, the final $r\times r$ reduced matrix, with $r=2n$, is formed
from products of $k_H\times r$ sketched quantities. This requires $\mathcal O(k_Hr^2)$ operations, since the
$r^2$ matrix entries are obtained from inner products of length $k_H$.
Therefore, double-sided sketching does not substantially reduce
the computational cost while introducing an additional approximation.

Although symmetrizing $M(\mu)$ would already yield a sketched reduced
Hamiltonian system, $V$ is generally not orthonormal with respect to the
sketched inner product induced by $\Theta_H$. Following the
sketched-inner-product framework of
\cite{Balabanov2019,Balabanov2021,Balabanov2019a}, we therefore introduce a
least-squares reconstruction that accounts for the sketched Gram matrix.

To make this reconstruction explicit, we define the orthogonal projector $P_V:=VV^\top$ and the operator residual
\begin{equation*}
R_H(\mu)
:=
(I_{2N}-P_V)H(\mu)V.
\end{equation*}
Because $V^\top V=I_{2n}$, the operator action admits the orthogonal
decomposition
\begin{equation}
\label{eq:operator_exact_decomp}
H(\mu)V
=
VH_r(\mu)+R_H(\mu),
\qquad
V^\top R_H(\mu)=0.
\end{equation}
Applying $\Theta_H$ yields
$\Theta_HH(\mu)V
=
(\Theta_HV)H_r(\mu)+\Theta_HR_H(\mu).$ Thus, $H_r(\mu)$ appears as the coefficient matrix of the component of
$H(\mu)V$ lying in $\operatorname{range}(V)$, while the sketched residual
accounts for the part orthogonal to the reduced space. We therefore define the least-squares problem
\begin{equation}
\label{eq:LS_core_problem}
H_{\mathrm{LS}}(\mu)
:=
\arg\min_{X\in\mathbb{R}^{2n\times2n}}
\|
\Theta_HVX-\Theta_HH(\mu)V
\|_F.
\end{equation}
If $\Theta_HV$ has full column rank, the minimizer is unique. Introducing
the sketched Gram matrix
\begin{equation*}
G
:=
(\Theta_HV)^\top(\Theta_HV),
\end{equation*}
the least-squares solution can be written as
\begin{equation*}
H_{\mathrm{LS}}(\mu)
=
G^{-1}
(\Theta_HV)^\top\Theta_HH(\mu)V.
\end{equation*}
In practice, $G^{-1}$ is not formed explicitly. Instead, the corresponding
linear systems may be solved using a Cholesky factorization of $G$, or
\eqref{eq:LS_core_problem} may be solved directly by QR factorization.

The \emph{sketched reduced Hamiltonian} is then defined by
\begin{equation*}
\widehat{\mathcal{H}}_r(\hat x_r;\mu)
:=
\tfrac12\hat x_r^\top\widehat H_r(\mu)\hat x_r,
\qquad
\widehat H_r(\mu)
:=
\operatorname{sym}\!\bigl(H_{\mathrm{LS}}(\mu)\bigr).
\end{equation*}
The corresponding sketched reduced Hamiltonian system is given by
\begin{equation}
\label{eq:sketchedHamSystemReduced}
\dot{\hat x}_r(t;\mu)
=
J_{2n}\widehat H_r(\mu) \hat x_r(t;\mu),
\qquad
\hat x_r(t_0;\mu)=V^+x_0(\mu).
\end{equation}
Because $\widehat H_r(\mu)$ is symmetric,
\eqref{eq:sketchedHamSystemReduced} is again a canonical Hamiltonian system.
Randomization therefore affects only the approximation of the reduced
Hamiltonian, not the Hamiltonian structure of the resulting ROM.

For fixed $V$, it is convenient to introduce $V^{\Theta_H}:=\Theta_HV,$ and $W^{\Theta_H}(\mu):=\Theta_HH(\mu)V.$
Then, the reduced quantities can be expressed entirely in terms of these
sketched matrices:
\[
G=(V^{\Theta_H})^\top V^{\Theta_H},
\qquad
M(\mu)=(V^{\Theta_H})^\top W^{\Theta_H}(\mu),
\qquad
\widehat H_r(\mu)
=
\operatorname{sym}\!\bigl(G^{-1}M(\mu)\bigr).
\]
For the parameter-separable decomposition \eqref{eq:param_sep_H_motivation}, all
high-dimensional operations can be confined to the offline stage. We first
compute the sketch of the reduced basis matrix $V^{\Theta_H}$ and the sketched Gram matrix $G$
and, for $q=1,\ldots,Q$, the operator sketches $W_q^{\Theta_H}
:=
\Theta_HH_qV$ and the sketched matrices $M_q
:=
(V^{\Theta_H})^\top W_q^{\Theta_H}.$ 
For each component, we solve $GX_q=M_q$
and define the sketched components
\begin{equation*}
\widehat H_{r,q}
:=
\operatorname{sym}(X_q)
=
\operatorname{sym}(G^{-1}M_q).
\end{equation*}
Linearity then gives
\begin{equation*}
\widehat H_r(\mu)
=
\sum_{q=1}^{Q}
\vartheta_q(\mu)\widehat H_{r,q}.
\end{equation*}
Once the reduced components have been precomputed, online assembly
therefore involves only reduced-dimensional quantities and is independent of
both the full-order dimension $N$ and the sketch dimension $k_H$.

For the sparse operators and CountSketch used in our experiments, this sketch
is accumulated directly during the application of $H_q$. In particular, each computed row of $H_qV$ is immediately added, with its random sign, to the corresponding output row. The full matrix
$H_qV\in\mathbb R^{2N\times2n}$ therefore never needs to be assembled or
stored. Once $W_q^{\Theta_H}$ has been formed, all subsequent products and
least-squares solves involve only the sketch and reduced dimensions. After
the reduced component $\widehat H_{r,q}$ has been assembled,
$W_q^{\Theta_H}$ can be discarded.

\subsection{Reduced-Hamiltonian perturbation analysis}
\label{sec:sketched_ham_perturbation}

We now bound the difference between $H_r(\mu)$ and
$\widehat H_r(\mu)$. By \eqref{eq:operator_exact_decomp}, the
least-squares reconstruction reproduces the component of $H(\mu)V$ in
$\operatorname{range}(V)$; the error is therefore governed by
$R_H(\mu)$.

\begin{lemma}[Uniform reduced-Hamiltonian perturbation bound]
\label{lem:delta_uniform_mu}
Let $\mathcal P$ be a parameter set,
$H(\mu)=H(\mu)^\top\in\mathbb R^{2N\times2N}$, and let
$V\in\mathbb R^{2N\times2n}$ be ortho-symplectic. Set
$H_r(\mu):=V^\top H(\mu)V$ and $P_V:=VV^\top$, and assume:
\begin{enumerate}
\item[(A1)] $H_r(\mu)$ is positive definite for all $\mu\in\mathcal P$,
with $\alpha:=\inf_{\mu\in\mathcal P}\lambda_{\min}(H_r(\mu))>0.$
\item[(A2)] $H(\mu)=\sum_{q=1}^Q\vartheta_q(\mu)H_q$, with
$H_q=H_q^\top$.
\item[(A3)]
$\vartheta_{\max}:=
\sup_{\mu\in\mathcal P}\sum_{q=1}^Q|\vartheta_q(\mu)|<\infty$.
\end{enumerate}
For each component, set
\[
R_q:=(I_{2N}-P_V)H_qV,
\qquad
R_H(\mu):=\sum_{q=1}^Q\vartheta_q(\mu)R_q,
\]
and define the uniform residual constant
\begin{equation*}
C_R:=
\sup_{\mu\in\mathcal P}
\sum_{q=1}^Q|\vartheta_q(\mu)|\|R_q\|_2
\le
\vartheta_{\max}\max_q\|R_q\|_2<\infty.
\end{equation*}
Consider the augmented space $Y_H:=\operatorname{range}[V,R_1,\ldots,R_Q]$ with $
d_H:=\dim(Y_H)\le2n(Q+1) $ and assume that
$\Theta_H\in\mathbb R^{k_H\times2N}$ is an
$\varepsilon$-subspace embedding for $Y_H$, with
$0\le\varepsilon<1$. Let
\[
G:=(\Theta_HV)^\top(\Theta_HV),\qquad
\widehat H_r(\mu):=
\operatorname{sym}\!\left(
G^{-1}(\Theta_HV)^\top\Theta_HH(\mu)V
\right),
\]
and set $E(\mu):=\widehat H_r(\mu)-H_r(\mu)$ and $\delta(\mu):=
\left\|
H_r(\mu)^{-1/2}E(\mu)H_r(\mu)^{-1/2}
\right\|_2.$ 
Then, for every $\mu\in\mathcal P$, we obtain
\begin{equation}
\label{eq:E_uniform_bound}
\|E(\mu)\|_2
\le
\frac{\varepsilon}{1-\varepsilon}\|R_H(\mu)\|_2
\le
\frac{\varepsilon}{1-\varepsilon}C_R.
\end{equation}
Moreover, it follows that
\begin{equation}
\label{eq:delta_uniform_bound}
\delta(\mu)
\le
\frac{\varepsilon}{1-\varepsilon}\frac{C_R}{\alpha},
\qquad
\mu\in\mathcal P
\end{equation}
and if \begin{equation}\label{eq:eps_uniform_choice}
\varepsilon <
\frac{\alpha}{\alpha+C_R} 
\end{equation}
we additionally obtain
$\sup_{\mu\in\mathcal P}\delta(\mu)<1.$
\end{lemma}

\begin{proof}
Since $\operatorname{range}(V)\subseteq Y_H$, the embedding property may be
applied to $Vz$ for every $z\in\mathbb R^{2n}$. Moreover,
$V^\top V=I_{2n}$ implies $\|Vz\|_2=\|z\|_2$, while
$z^\top Gz=\|\Theta_HVz\|_2^2$. Hence
\[
(1-\varepsilon)\|z\|_2^2
\le z^\top Gz
\le(1+\varepsilon)\|z\|_2^2.
\]
Thus $\lambda_{\min}(G)\ge1-\varepsilon>0$, so $G$ is positive definite
and
\begin{equation}
\label{eq:Ginv_bounds}
\|G^{-1}\|_2\le\frac{1}{1-\varepsilon}.
\end{equation}
Using \eqref{eq:operator_exact_decomp} in the least-squares reconstruction
gives
\[
H_{\mathrm{LS}}(\mu)
=
H_r(\mu)
+
G^{-1}(\Theta_HV)^\top\Theta_HR_H(\mu).
\]
Since $V^\top R_H(\mu)=0$, the reconstruction error can be written as
\begin{equation}
\label{eq:HLS_residual_identity}
H_{\mathrm{LS}}(\mu)-H_r(\mu)
=
G^{-1}V^\top
(\Theta_H^\top\Theta_H-I_{2N})R_H(\mu).
\end{equation}
To bound the second factor, we take $z,w\in\mathbb R^{2n}$. Both
$Vz$ and $R_H(\mu)w$ belong to $Y_H$, while
$(Vz)^\top R_H(\mu)w=0$. Hence
Lemma~\ref{lem:ip_from_embedding} gives
\[
\begin{aligned}
\left|
z^\top V^\top
(\Theta_H^\top\Theta_H-I_{2N})R_H(\mu)w
\right|
\le
\varepsilon\|Vz\|_2\|R_H(\mu)w\|_2 \le
\varepsilon\|z\|_2\|R_H(\mu)\|_2\|w\|_2.
\end{aligned}
\]
Taking the supremum over unit vectors $z$ and $w$ therefore yields
\[
\left\|
V^\top(\Theta_H^\top\Theta_H-I_{2N})R_H(\mu)
\right\|_2
\le
\varepsilon\|R_H(\mu)\|_2.
\]
Together with \eqref{eq:Ginv_bounds} and
\eqref{eq:HLS_residual_identity}, this shows $\|H_{\mathrm{LS}}(\mu)-H_r(\mu)\|_2
\le
\frac{\varepsilon}{1-\varepsilon}\|R_H(\mu)\|_2.$
Since $H_r(\mu)$ is symmetric and
$
E(\mu)
=
\operatorname{sym}
\bigl(H_{\mathrm{LS}}(\mu)-H_r(\mu)\bigr),
$
the nonexpansiveness of symmetrization in the spectral norm gives \eqref{eq:E_uniform_bound}. Finally, $\lambda_{\min}(H_r(\mu))\ge\alpha$ implies
$\|H_r(\mu)^{-1}\|_2\le\alpha^{-1}$, and hence we obtain \eqref{eq:delta_uniform_bound}. Condition
\eqref{eq:eps_uniform_choice} makes the final bound strictly smaller than
one and therefore gives
$\sup_{\mu\in\mathcal P}\delta(\mu)<1$.
\end{proof}

\begin{remark}[Exactness on invariant reduced spaces]
\label{rem:sketched_ham_invariant_exact}
If $H(\mu)V\subseteq\operatorname{range}(V),$ 
then $R_H(\mu)=0$ and therefore $H_{\mathrm{LS}}(\mu)
=
\widehat H_r(\mu)
=
H_r(\mu).$ Thus, the sketch introduces no reduced-Hamiltonian error on an invariant
reduced space.
\end{remark}

The lemma is a deterministic result under the stated embedding assumption. Since $\Theta_H$ is independent of the randomness used to construct $V$,
the realized basis matrix $V$, and hence the space $Y_H$, is fixed with respect
to the randomness in $\Theta_H$. Standard oblivious subspace-embedding results \cite[Section~3.2]{Balabanov2019} therefore apply directly to
$Y_H$. For CountSketch, SRHT, and SRFT embeddings, standard oblivious subspace-embedding results ensure the required \(\varepsilon\)-subspace embedding with probability at least \(1-\delta_{\mathrm{emb}}\), provided that the sketch dimension \(k_H\) is chosen sufficiently large; see, e.g., \cite{Woodruff2014,Tropp2011Ana}.

\subsection{Energy-norm trajectory error}
\label{sec:sketched_ham_trajectory_error}

In this section, we convert the reduced-Hamiltonian perturbation into a trajectory-error
bound. For positive definite quadratic Hamiltonians, energy conservation
avoids the exponential-in-time factor of a generic linear-system
perturbation estimate. We omit the parameter dependence here to simplify the notation and improve readability. Our argument follows a similar energy-norm variation-of-constants argument as in the port-Hamiltonian error analysis of \cite{RettbergEtAl2024}, exploiting
the Hamiltonian structure to control the propagator in the energy norm.

\begin{theorem}[Energy-norm error bound for a sketched quadratic Hamiltonian]
\label{thm:energy_error_linearHam}
Let
$H_r\in\mathbb{R}^{2n\times2n}$ be symmetric positive definite and consider a reduced linear Hamiltonian system as in  \eqref{eq:red-lin-Ham_sys}.
Let $\widehat H_r=\widehat H_r^\top$, set $E:=\widehat H_r-H_r,$ and assume $\delta
=
\left\|
H_r^{-1/2}EH_r^{-1/2}
\right\|_2
<
1.$

Then $\widehat H_r$ is positive definite. If
$\widehat x_r(t)$ solves the sketched reduced system \eqref{eq:sketchedHamSystemReduced} then, for all $t\ge t_0$,
\begin{equation}
\label{eq:energy_error_bound}
\|
\widehat x_r(t)-x_r(t)
\|_{H_r}
\le
(t-t_0)
\Bigl\|
H_r^{1/2}J_{2n}H_r^{1/2}
\Bigr\|_2
\delta
\sqrt{\frac{1+\delta}{1-\delta}}
\,
\|x_{r,0}\|_{H_r},
\end{equation}
where
\[
\|y\|_{H_r}:=\sqrt{y^\top H_ry}, \qquad y\in \R^{2n}.
\]
\end{theorem}

\begin{proof}
We set $A:=J_{2n}H_r,$ and $e(t):=\widehat x_r(t)-x_r(t).$
Subtracting the exact \eqref{eq:red-lin-Ham_sys} and perturbed \eqref{eq:sketchedHamSystemReduced} reduced systems gives the error equation
\[
\dot e(t)
=
Ae(t)+J_{2n}E\widehat x_r(t),
\qquad
e(t_0)=0.
\]
Variation of constants therefore yields
\begin{equation}
\label{eq:VOC}
e(t)
=
\int_{t_0}^{t}
\Phi(t)\Phi(s)^{-1}
J_{2n}E\widehat x_r(s)\,\mathrm ds,
\end{equation}
where $\Phi$ denotes the fundamental solution of the unperturbed system
$\dot y=Ay$.

The Hamiltonian structure implies that the unperturbed flow preserves the
$H_r$-energy norm. Indeed, since $A^\top H_r+H_rA=0$, the fundamental solution satisfies
$\Phi(t)^\top H_r\Phi(t)=H_r$ and therefore
$\|\Phi(t)y\|_{H_r}=\|y\|_{H_r}$ for all $y\in\mathbb R^{2n}.$
From that it follows that $\|\Phi(t)\Phi(s)^{-1}y\|_{H_r}
=
\|y\|_{H_r}$ for all $y\in\mathbb R^{2n}.$
 Taking the $H_r$-norm in
\eqref{eq:VOC} therefore gives
\begin{equation}
\label{eq:bound1}
\|e(t)\|_{H_r}
\le
\int_{t_0}^{t}
\|J_{2n}E\widehat x_r(s)\|_{H_r}\,\mathrm ds.
\end{equation}
For any $y\in\mathbb R^{2n}$, we have
\[
\|J_{2n}Ey\|_{H_r}
=
\|H_r^{1/2}J_{2n}Ey\|_2
\le
\|
H_r^{1/2}J_{2n}EH_r^{-1/2}
\|_2
\|y\|_{H_r}.
\]
Applying this estimate to \eqref{eq:bound1} yields
\begin{equation}
\label{eq:bound2}
\|e(t)\|_{H_r}
\le
\|
H_r^{1/2}J_{2n}EH_r^{-1/2}
\|_2
\int_{t_0}^{t}
\|\widehat x_r(s)\|_{H_r}\,\mathrm ds.
\end{equation}
Next we define $K:=H_r^{-1/2}EH_r^{-1/2}$, which implies 
$\delta = \|K\|_2.$ Since we have $E=H_r^{1/2}KH_r^{1/2},$
we obtain
\[
H_r^{1/2}J_{2n}EH_r^{-1/2}
=
\bigl(H_r^{1/2}J_{2n}H_r^{1/2}\bigr)K,
\]
and therefore arrive at
\begin{equation}
\label{eq:HJEBH_bound_delta}
\|
H_r^{1/2}J_{2n}EH_r^{-1/2}
\|_2
\le
\|
H_r^{1/2}J_{2n}H_r^{1/2}
\|_2\,\delta.
\end{equation}
It remains to control the perturbed trajectory in the $H_r$-norm.
Since $ H_r^{-1/2}\widehat H_rH_r^{-1/2}
=
I+K$ and $\|K\|_2=\delta<1$, the eigenvalues of $I+K$ lie in
$[1-\delta,1+\delta]$. Thus $\widehat H_r$ is positive definite and
\[
(1-\delta)\|y\|_{H_r}^2
\le
\|y\|_{\widehat H_r}^2
\le
(1+\delta)\|y\|_{H_r}^2
\qquad
\forall\,y\in\mathbb R^{2n}.
\]
This means that the perturbed system is itself Hamiltonian with symmetric positive definite
Hessian $\widehat H_r$, so it conserves the corresponding energy norm:
\[
\|\widehat x_r(s)\|_{\widehat H_r}
=
\|x_{r,0}\|_{\widehat H_r}.
\]
Using the norm equivalence above first for $\widehat x_r(s)$ and then for
the initial value gives
\[
\|\widehat x_r(s)\|_{H_r}
\le
\sqrt{\frac{1+\delta}{1-\delta}}
\|x_{r,0}\|_{H_r}.
\]
Substituting this estimate and
\eqref{eq:HJEBH_bound_delta} into \eqref{eq:bound2} yields \eqref{eq:energy_error_bound}.
\end{proof}

The estimate grows at most linearly in time because the unperturbed
Hamiltonian flow is an $H_r$-isometry. The factor $\omega_r
:=
\Bigl\|
H_r^{1/2}J_{2n}H_r^{1/2}
\Bigr\|_2$ is the largest frequency scale of the reduced dynamics in energy
coordinates, so the error is governed by $\omega_r$, the relative
Hamiltonian perturbation $\delta$, and elapsed time.

\subsection{High-probability uniform-in-parameter error guarantee}
\label{sec:sketched_ham_high_probability}

We finally combine the perturbation, embedding, and trajectory estimates. In addition to the assumptions of Lemma~\ref{lem:delta_uniform_mu}, we assume
\begin{equation}
\label{eq:beta_uniform}
\beta
:=
\sup_{\mu\in\mathcal P}
\|H_r(\mu)\|_2
<
\infty, \qquad
\sup_{\mu\in\mathcal P}
\|x_{r,0}(\mu)\|_{H_r(\mu)}
<
\infty.
\end{equation}
\begin{corollary}[High-probability uniform-in-parameter energy-norm error bound]
\label{cor:hp_uniform_energy_error}
Adopt the setting of Lemma~\ref{lem:delta_uniform_mu} and assume
\eqref{eq:beta_uniform} holds. Let
\(\delta_{\mathrm{emb}}\in(0,1)\), and
\[
\varepsilon<\frac{\alpha}{\alpha+C_R}.
\]
Suppose \(k_H\) is chosen such that \(\Theta_H\) is an
\(\varepsilon\)-subspace embedding for
\(Y_H=\operatorname{range}[V,R_1,\ldots,R_Q]\)
with probability at least \(1-\delta_{\mathrm{emb}}\), and define
\[
\delta_*:=
\frac{\varepsilon}{1-\varepsilon}\frac{C_R}{\alpha}<1.
\]
Then, with probability at least \(1-\delta_{\mathrm{emb}}\),
\[\delta(\mu)\le\delta_*<1\qquad\forall\,\mu\in\mathcal P,
\]
and hence \(\widehat H_r(\mu)\) is positive definite for every
\(\mu\in\mathcal P\).

If \(\widehat x_r(t;\mu)\) solves the sketched reduced Hamiltonian system \eqref{eq:sketchedHamSystemReduced}
then, for every \(\mu\in\mathcal P\) and
\(t\in[t_0,t_{\textrm{end}}]\),
\begin{equation}
\label{eq:hp_uniform_energy_error}
\|
\widehat x_r(t;\mu)-x_r(t;\mu)
\|_{H_r(\mu)}
\le
(t-t_0)\,
\beta\,
\delta_*
\sqrt{\frac{1+\delta_*}{1-\delta_*}}
\,
\sup_{\mu\in\mathcal P} \|x_{r,0}(\mu)\|_{H_r(\mu)}.
\end{equation}
\end{corollary}

\begin{proof}
We consider a realization for which \(\Theta_H\) is an
\(\varepsilon\)-subspace embedding for \(Y_H\). By Lemma~\ref{lem:delta_uniform_mu}, we have $\delta(\mu)
\le
\delta_*
<
1$ for all $\mu\in\mathcal P,$
and hence \(\widehat H_r(\mu)\) is positive definite for every
\(\mu\in\mathcal P\). For any fixed \(\mu\in\mathcal P\), Theorem~\ref{thm:energy_error_linearHam}
therefore applies and gives
\[
\|
\widehat x_r(t;\mu)-x_r(t;\mu)
\|_{H_r(\mu)}
\le
(t-t_0)
\Bigl\|
H_r(\mu)^{1/2}J_{2n}H_r(\mu)^{1/2}
\Bigr\|_2
f\bigl(\delta(\mu)\bigr)
\,
\|x_{r,0}(\mu)\|_{H_r(\mu)},
\]
where
\[
f(\delta)
=
\delta\sqrt{\frac{1+\delta}{1-\delta}},
\qquad 0\le\delta<1.
\]
Since it holds that $\Bigl\|
H_r(\mu)^{1/2}J_{2n}H_r(\mu)^{1/2}
\Bigr\|_2
\le
\|H_r(\mu)\|_2
\le
\beta,$ and \(f\) is increasing on \([0,1)\), the bound
\(\delta(\mu)\le\delta_*\) yields \eqref{eq:hp_uniform_energy_error}.
Since the embedding property holds with probability at least
\(1-\delta_{\mathrm{emb}}\), the result follows.
\end{proof}

Thus, the sketch dimension $k_H$ controls a high-probability trajectory
bound that holds uniformly over parameters and the prescribed time interval.
The estimate depends on the reduced coercivity $\alpha$, operator residual
constant $C_R$, spectral scale $\beta$, the uniform initial-energy bound, and
embedding accuracy $\varepsilon$, while the sketched ROM remains
Hamiltonian.

\section{Numerical experiments}
\label{Sec:NumExp}

We investigate the two randomized components introduced in this work,
separately and within the complete offline pipeline. \Cref{sec:numexp-models}
introduces the linear Schr\"odinger equation and lossless LC resonator network.
\Cref{sec:numexp-ops} compares Ops-rcSVD with deterministic and randomized
basis-construction methods, while \Cref{sec:numexp-sketched-rom} combines
Ops-rcSVD with the sketched reduced-Hamiltonian assembly and assesses the
resulting ROM accuracy. 

\subsection{Model problems}
\label{sec:numexp-models}

We consider two linear Hamiltonian systems with quadratic,
parameter-dependent Hamiltonians: a finite-difference discretization of the
linear Schr\"odinger equation and a lossless network of coupled LC resonators.

All experiments were performed on a system with $64\,\mathrm{GB}$ RAM and
a 13th Gen Intel Core i7-13700K processor. Further implementation details are
provided with the source code which is available at Darus\footnote{https://doi.org/10.18419/DARUS-6508}.

\subsubsection{Linear Schr\"odinger equation}
\label{sec:numexp-model-schrodinger}

We consider a parameter-dependent linear Schr\"odinger equation. Such equations
are standard benchmarks for structure-preserving MOR; see, e.g.,
\cite{MaboudiAfkham2017,KarasozenUzunca2018}. Here we use a two-dimensional problem with a parameter-dependent external potential and omit the nonlinear interaction term:
\begin{equation*}
\mathrm{i}\,\partial_t \psi(\xi_1,\xi_2,t;\mu)
=
\left(
-\frac{1}{2}\Delta+\mu W(\xi_2)
\right)
\psi(\xi_1,\xi_2,t;\mu),
\qquad
(\xi_1,\xi_2)\in(0,1)^2,
\end{equation*}
with periodic conditions at \(\xi_1\in\{0,1\}\) and homogeneous Dirichlet conditions at \(\xi_2\in\{0,1\}\), i.e.,
\[
\psi(0,\xi_2,t;\mu)=\psi(1,\xi_2,t;\mu),
\quad
\partial_{\xi_1}\psi(0,\xi_2,t;\mu)=\partial_{\xi_1}\psi(1,\xi_2,t;\mu),
\quad
\psi(\xi_1,0,t;\mu)=\psi(\xi_1,1,t;\mu)=0,
\]
and parameter domain $\mathcal P=[1,2].$ Using
\[
b(s)
=
\begin{cases}
\exp\!\left(-\dfrac{1}{1-s^2}\right), & |s|<1,\\
0, & |s|\ge1,
\end{cases}
\]
we define
\begin{equation*}
\psi_0(\xi_1,\xi_2)
=
c_0\,b\!\left(\frac{\xi_2-0.30}{0.16}\right),
\qquad
W(\xi_2)
=
\frac{
b\!\left(\dfrac{\xi_2-0.60}{0.14}\right)
}{
\left\|
b\!\left(\dfrac{\cdot-0.60}{0.14}\right)
\right\|_{L^\infty(0,1)}
},
\end{equation*}
where $c_0$ normalizes $\psi_0$ in $L^2((0,1)^2)$.

Using second-order central differences on a tensor-product grid, with
\(N=N_{\xi_1}N_{\xi_2}\), we obtain
\begin{equation*}
L
=
I_{N_{\xi_2}}\otimes L_{\xi_1}
+
L_{\xi_2}\otimes I_{N_{\xi_1}},
\qquad
L_{\xi_i}
=
\frac{1}{h_{\xi_i}^2}\operatorname{tridiag}(-1,2,-1),
\quad i=1,2,
\end{equation*}
where \(L_{\xi_1}\) additionally contains the entries associated
with the periodic boundary conditions. The discrete operator is
\begin{equation*}
A(\mu)=\frac12 L+\mu W_h,
\qquad
W_h=\operatorname{diag}(w_h),
\end{equation*}
where \(w_h\) contains the nodal values of \(W(\xi_2)\), repeated along
the \(\xi_1\)-direction.

Writing $\psi_h=q+\mathrm{i}p$ and
$x=[q^\top,p^\top]^\top$ gives
\begin{equation*}
\dot x(t;\mu)
=
J_{2N}H(\mu)x(t;\mu),
\qquad
H(\mu)
=
\begin{bmatrix}
A(\mu)&0\\
0&A(\mu)
\end{bmatrix},
\end{equation*}
with parameter-separable decomposition
\begin{equation*}
H(\mu)
=
H_1+\mu H_2,
\qquad
H_1
=
\frac12
\begin{bmatrix}
L&0\\
0&L
\end{bmatrix},
\qquad
H_2
=
\begin{bmatrix}
W_h&0\\
0&W_h
\end{bmatrix}.
\end{equation*}
For time discretization we use the implicit midpoint rule with $\Delta t=10^{-3}$ and set the parameter train and test domains as $\mathcal P_{\mathrm{train}} = \{1,1.2,1.4,1.6,1.8,2\},$ and $\mathcal P_{\mathrm{test}} = \{1.3,1.5,1.7\}.$
Spatial resolution and final time are specified separately for each study.

Since $L_{\xi_2}$ is positive definite and $W_h$ positive semidefinite,
$H(\mu)$ is positive definite for all $\mu\in \mathcal{P}$, satisfying the
assumptions of
\Cref{sec:sketched_ham_perturbation,sec:sketched_ham_trajectory_error}.

\subsubsection{Lossless LC resonator network}
\label{sec:numexp-model-lc}

Our second example is a lossless rectangular network of coupled LC
resonators. Electrical networks are classical benchmarks for port-Hamiltonian
MOR \cite{PolyugaVanderSchaft2010,GugercinEtAl2012,ChaturantabutEtAl2016};
here we use a closed, lossless two-dimensional network with parameter-dependent anisotropic coupling.

We denote the node fluxes and capacitor charges by
\(\phi,q\in\mathbb R^N\), respectively, where
\(N=N_{\xi_1}N_{\xi_2}\). In normalized units, the
capacitive and local inductive energies are
\(\frac12 q^\top q\) and \(\frac12\phi^\top\phi\). Additional inductive
couplings connect neighboring resonators and depend on the differences of
their node fluxes.

We first define
\[
B_{\xi_1}
=
I_{N_{\xi_2}}\otimes D_{N_{\xi_1}},
\quad
B_{\xi_2}
=
D_{N_{\xi_2}}\otimes I_{N_{\xi_1}} \quad
\textrm{with } 
D_m=
\begin{bmatrix}
-1 & 1 \\
   & -1 & 1 \\
   &    & \ddots & \ddots \\
   &    &        & -1 & 1
\end{bmatrix}
\in\mathbb R^{(m-1)\times m}
\]
such that \(B_{\xi_1}\phi\) and \(B_{\xi_2}\phi\) contain the flux differences
between neighboring resonators in the two coordinate directions. We choose
the corresponding normalized inverse inductances as $c_{\xi_1}(\mu)=\mu,
$ and $c_{\xi_2}(\mu)=0.05+0.6\mu,$ where $\mu\in\mathcal P:=[0.1,0.6].$
The total inductive energy is therefore
\[
\frac12\|\phi\|_2^2
+
\frac{\mu}{2}\|B_{\xi_1}\phi\|_2^2
+
\frac{0.05+0.6\mu}{2}\|B_{\xi_2}\phi\|_2^2
=
\frac12\phi^\top K(\mu)\phi,
\]
with 
\[
K(\mu)
:=
I_N+0.05L_{\xi_2}
+\mu\bigl(L_{\xi_1}+0.6L_{\xi_2}\bigr)
\]
where $L_{\xi_1}:=B_{\xi_1}^\top B_{\xi_1},$ and $L_{\xi_2}:=B_{\xi_2}^\top B_{\xi_2}.$ The Hamiltonian of the full coupled system is therefore
\[
\mathcal H(\phi,q;\mu)
=
\frac12\phi^\top K(\mu)\phi
+
\frac12q^\top q.
\]
Introducing \(x=[\phi^\top,q^\top]^\top\) gives
\[
\dot x(t;\mu)
=
J_{2N}H(\mu)x(t;\mu),
\qquad
H(\mu)
=
\begin{bmatrix}
K(\mu)&0\\
0&I_N
\end{bmatrix}.
\]
The Hamiltonian matrix is parameter-separable, i.e., $H(\mu)=H_0+\mu H_1,$
with
\[
H_0
=
\begin{bmatrix}
I_N+0.05L_{\xi_2}&0\\
0&I_N
\end{bmatrix},
\qquad
H_1
=
\begin{bmatrix}
L_{\xi_1}+0.6L_{\xi_2}&0\\
0&0
\end{bmatrix}.
\]
The initial condition consists of spatially separated Gaussian charge and flux
pulses centered at
\[
(i_q,j_q)=(0.42N_{\xi_1},0.57N_{\xi_2}),
\qquad
(i_\phi,j_\phi)=(0.68N_{\xi_1},0.31N_{\xi_2}),
\]
with
\[
q_{ij}^{(0)}
=
\exp\!\left(
-\frac{(i-i_q)^2+(j-j_q)^2}{2(3)^2}
\right),
\qquad
\phi_{ij}^{(0)}
=
0.5\exp\!\left(
-\frac{(i-i_\phi)^2+(j-j_\phi)^2}{2(2.5)^2}
\right).
\]
The combined state
\(x_0=[(\phi^{(0)})^\top,(q^{(0)})^\top]^\top\) is normalized such that
\(\|x_0\|_2=1\).

We use the implicit midpoint rule with \(\Delta t=0.01\) and set the parameter train and test domains as $\mathcal P_{\mathrm{train}}
=
\{0.1,0.2,0.3,0.4,0.5,0.6\} $ and $
\mathcal P_{\mathrm{test}}
=
\{0.25,0.35,0.45\}.$ The spatial resolution and final time are specified separately for each
numerical study.
\subsection{Ops-rcSVD basis construction}
\label{sec:numexp-ops}

The first study compares cSVD, cSVDs (computing a subset of dominant singular values and vectors), conventional two-pass
rcSVD, and Ops-rcSVD. The cSVD computes the full SVD of the complex snapshot matrix once and truncates it to each target dimension. The cSVDs instead computes only the prescribed dominant singular subspace and is run separately for each rank, providing a more competitive deterministic low-rank baseline. Problem sizes are chosen so that the complete snapshot matrices remain storable, allowing deterministic cSVD comparisons.

For an ortho-symplectic basis
$V\in\mathbb R^{2N\times2n}$, approximation quality is measured by
\begin{equation*}
e_{\mathrm{proj}}(V)
=
\frac{\|X_s-VV^\top X_s\|_F}{\|X_s\|_F}.
\end{equation*}
By \eqref{eq:real_complex_proj_identity}, this equals the complex projection
error analyzed in \Cref{sec:ops_rcsvd_error}. Results are plotted against the
real reduced dimension $2n$.

Reported runtimes include generation of the required training trajectories
and basis construction. The one-time setup of the parameter-dependent
midpoint propagators or factorizations is performed before the benchmark and
is excluded. For cSVD and cSVDs, the complete snapshot matrix is retained;
rcSVD and Ops-rcSVD process snapshots in streaming blocks. The second rcSVD
pass therefore regenerates all training trajectories, whereas Ops-rcSVD forms
both sketches during the first pass.

Both randomized methods use SRFT-based maps. We set $p=10, \ell=n+p,b=16,$ 
and for Ops-rcSVD we choose the practical heuristic $k=2\ell$, with $k\le N$.

For each target rank, five independent randomized realizations are used.
Projection errors are summarized by mean, and runtimes
by median since the median is less sensitive to outliers. Timing is single-threaded.

\subsubsection{Linear Schr\"odinger equation}
\label{sec:numexp-ops-schrodinger}

For this comparison, we set $N_{\xi_1}=N_{\xi_2}=600,$
which yields $N=360\,000$, $2N=720\,000$, and choose
$t\in[0,0.25]$. The six training trajectories yield approximately
$1.5\times10^3$ snapshots. The complex matrix
$Y_s\in\mathbb C^{N\times n_s}$
requires approximately $8.7\,\mathrm{GB}$ in complex double precision;
the equivalent real matrix requires the same storage.
Thus all four methods fit within the available memory.

Figure~\ref{fig:ops-schrodinger} shows nearly indistinguishable projection
errors for all methods. The error decreases from order $10^{-1}$ for the
smallest bases to approximately $10^{-2}$ at intermediate dimensions and
then drops rapidly towards machine precision. In particular, the nearly
coincident rcSVD and Ops-rcSVD curves indicate that the additional one-pass
coefficient-recovery error is negligible here.

Runtime differences are more pronounced. The full cSVD requires roughly
$4\text{--}5\times10^2\,\mathrm{s}$ whereas
the cSVDs is cheaper for small dimensions but becomes increasingly expensive,
eventually exceeding $10^3\,\mathrm{s}$. The rcSVD requires approximately
$3\times10^2\,\mathrm{s}$, while Ops-rcSVD is consistently fastest at about
$1.5\times10^2\,\mathrm{s}$ for smaller bases with only moderate growth.
Thus, Ops-rcSVD achieves essentially the same projection accuracy while
avoiding the second training-trajectory generation.

The memory comparison shows a substantial advantage of the randomized
methods. The additional peak memory is approximately $28\,\mathrm{GB}$
for cSVD and $17\text{--}18\,\mathrm{GB}$ for cSVDs, whereas both
randomized methods require less than $4\,\mathrm{GB}$ over the tested
dimensions. Their memory consumption increases with the basis dimension,
with Ops-rcSVD requiring slightly more memory than rcSVD due to the
additional left sketch retained during the single snapshot pass.

\begin{figure}[htbp]
    \centering

    \begin{minipage}[t]{0.48\textwidth}
        \centering
        \includegraphics[width=\linewidth]
        {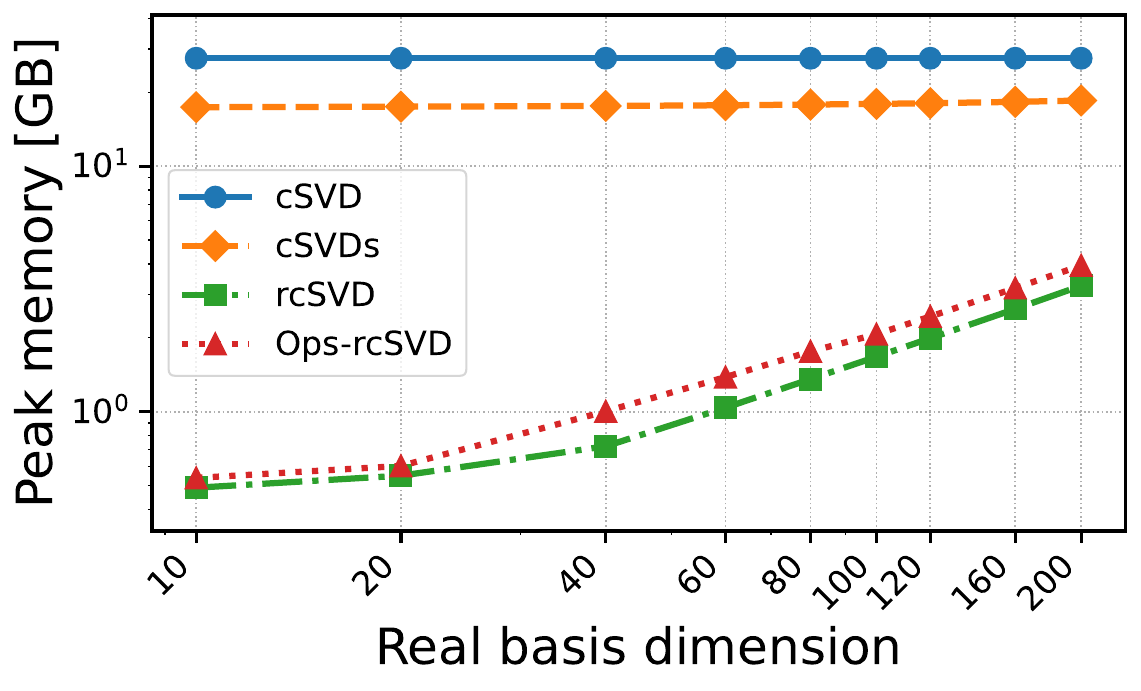}
        \par\smallskip
        \textnormal{(a) Peak memory consumption.}
    \end{minipage}
    \hfill
    \begin{minipage}[t]{0.48\textwidth}
        \centering
        \includegraphics[width=\linewidth]
        {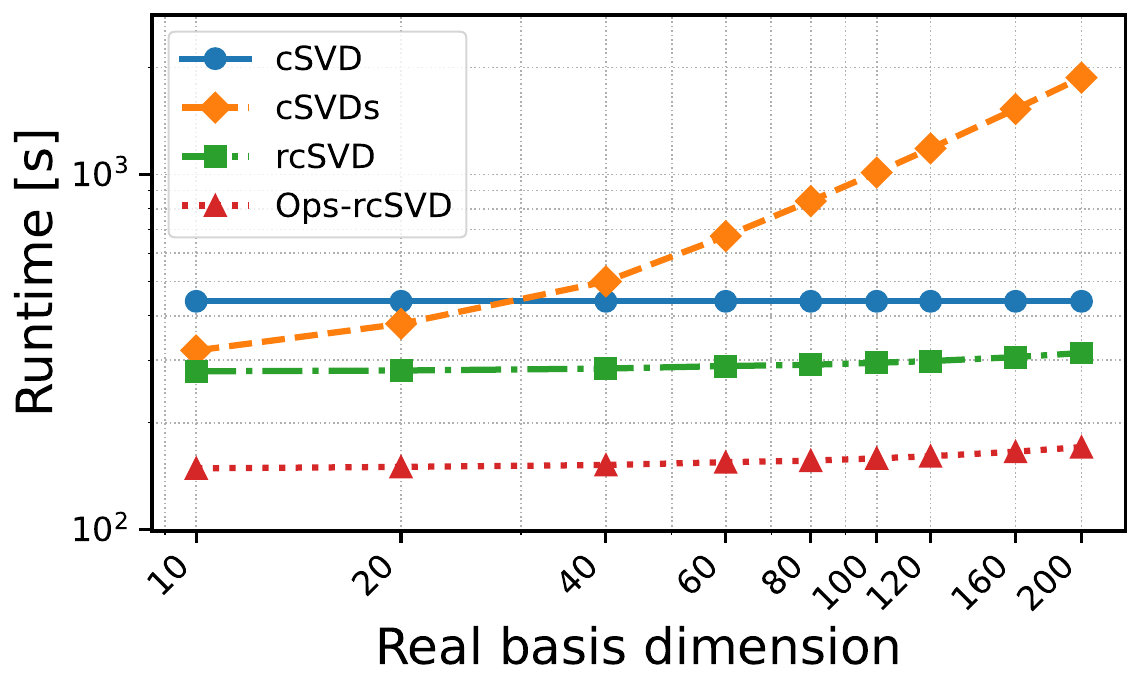}
        \par\smallskip
        \textnormal{(b) Snapshot generation and basis-construction time.}
    \end{minipage}

    \vspace{0.8em}

    \begin{minipage}[t]{0.48\textwidth}
        \vspace{0pt}
        \centering
        \includegraphics[width=\linewidth]
        {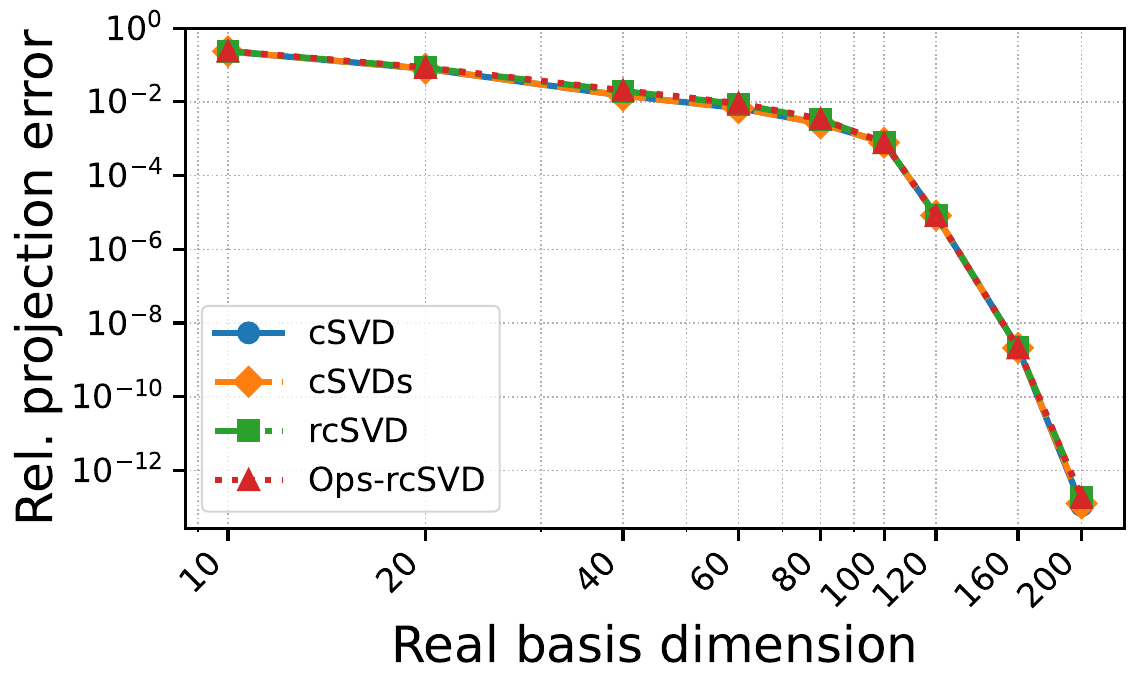}
        \par\smallskip
        \textnormal{(c) Relative projection error.}
    \end{minipage}
    \hfill
    \begin{minipage}[t]{0.48\textwidth}
        \vspace{0pt}
        \caption{Basis-construction results for the linear
        Schr\"odinger problem.
        (a) peak memory consumption,
        (b) snapshot generation and basis-construction time,
        and (c) relative projection error.
        Randomized errors show the mean;
        runtimes and memory consumption show the
        median over five realizations.}
        \label{fig:ops-schrodinger}
    \end{minipage}

\end{figure}

\subsubsection{Lossless LC resonator network}
\label{sec:numexp-ops-lc}

For the LC comparison, we set $N_{\xi_1}=N_{\xi_2}=100,$
which yields $N=10\,000$, $2N=20\,000$, and choose $t\in[0,120]$.
The six training trajectories produce approximately
$7.2\times10^4$ snapshots. The complex snapshot matrix $Y_s$
requires about $11.5\,\mathrm{GB}$, as does its real equivalent, so the
complete ensemble again fits in memory.

As shown in \Cref{fig:ops-lc}, all methods yield very similar projection
errors, decreasing from order $10^{-1}$ to near machine precision. In
particular, rcSVD and Ops-rcSVD remain essentially coincident, again
indicating that the one-pass coefficient recovery introduces no visible
loss of projection accuracy.

The runtime comparison shows the same qualitative behavior as for the
Schr\"odinger problem. The full cSVD has a large
cost, while cSVDs is cheaper at small ranks but grows strongly and eventually
approaches or exceeds the cSVD cost. Both randomized constructions remain
cheaper over most of the tested range, with Ops-rcSVD consistently fastest.
Since rcSVD and Ops-rcSVD have essentially identical approximation quality,
the runtime difference isolates the practical benefit of removing the second
trajectory pass.

The memory results also follow the behavior observed for the
Schr\"odinger problem. Both deterministic methods require substantially
more memory than the randomized methods, with nearly constant consumption
due to storage of the full snapshot ensemble. In contrast, the memory
consumption of rcSVD and Ops-rcSVD increases with the basis dimension but
remains considerably smaller. Ops-rcSVD again requires slightly more memory
than rcSVD, reflecting the additional storage needed for the left sketch.

\begin{figure}[htbp]
    \centering

    \begin{minipage}[t]{0.48\textwidth}
        \centering
        \includegraphics[width=\linewidth]
        {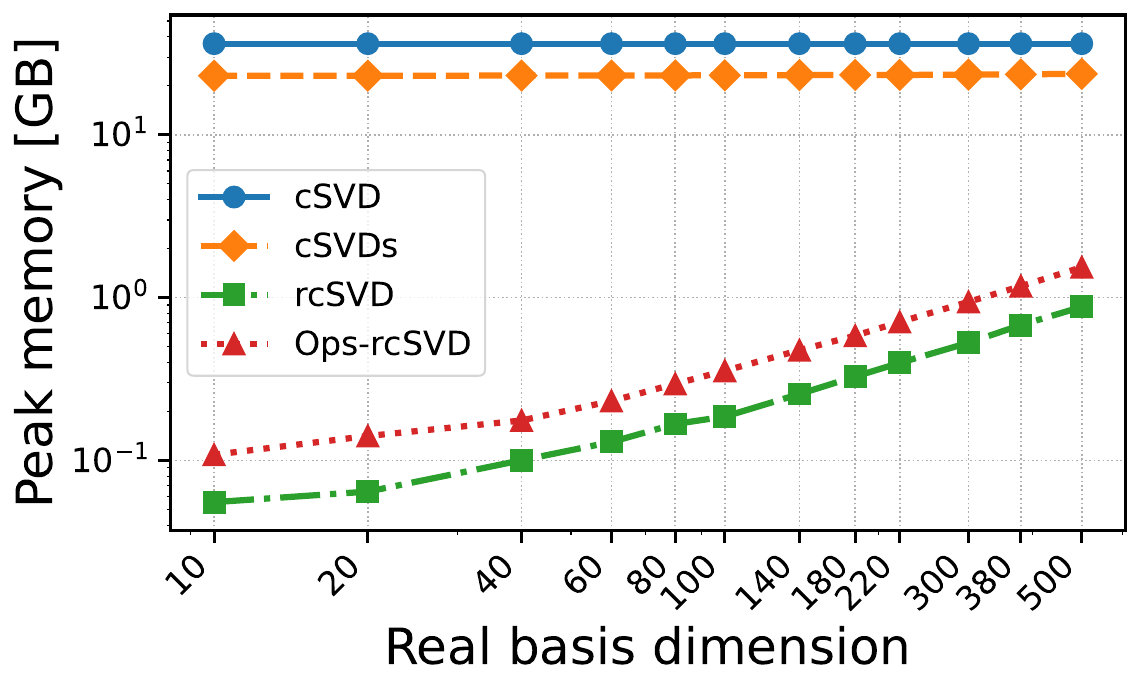}
        \par\smallskip
        \textnormal{(a) Peak memory consumption.}
    \end{minipage}
    \hfill
    \begin{minipage}[t]{0.48\textwidth}
        \centering
        \includegraphics[width=\linewidth]
        {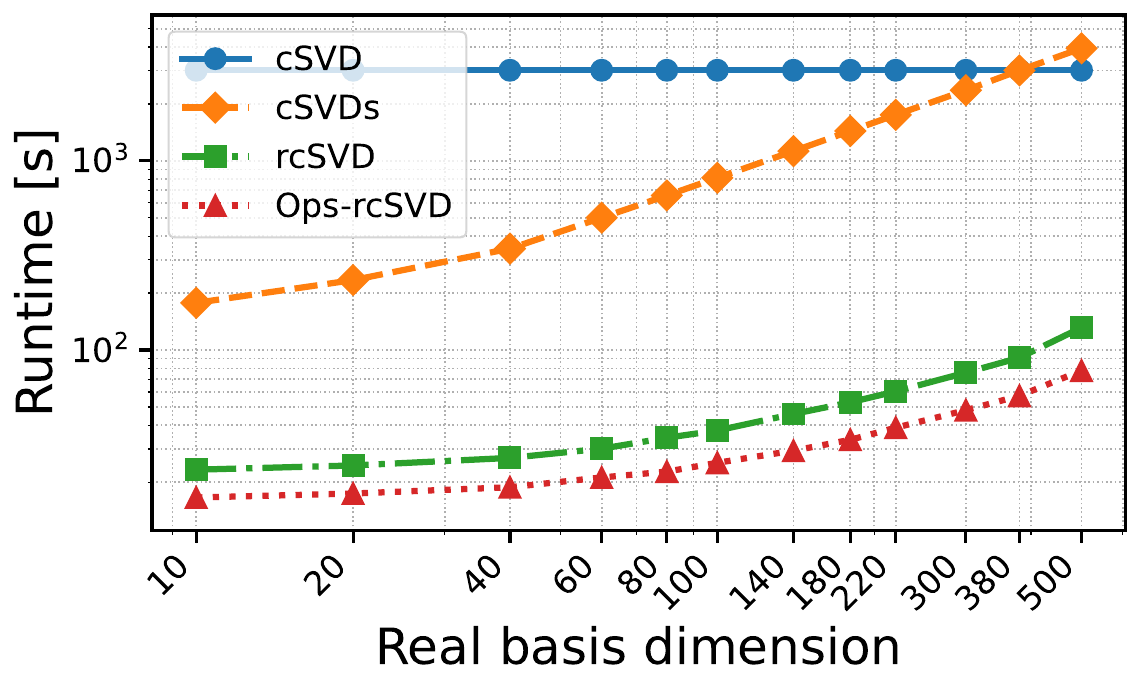}
        \par\smallskip
        \textnormal{(b) Snapshot generation and basis-construction time.}
    \end{minipage}

    \vspace{0.8em}

    \begin{minipage}[t]{0.48\textwidth}
    \vspace{0pt}
        \centering
        \includegraphics[width=\linewidth]
        {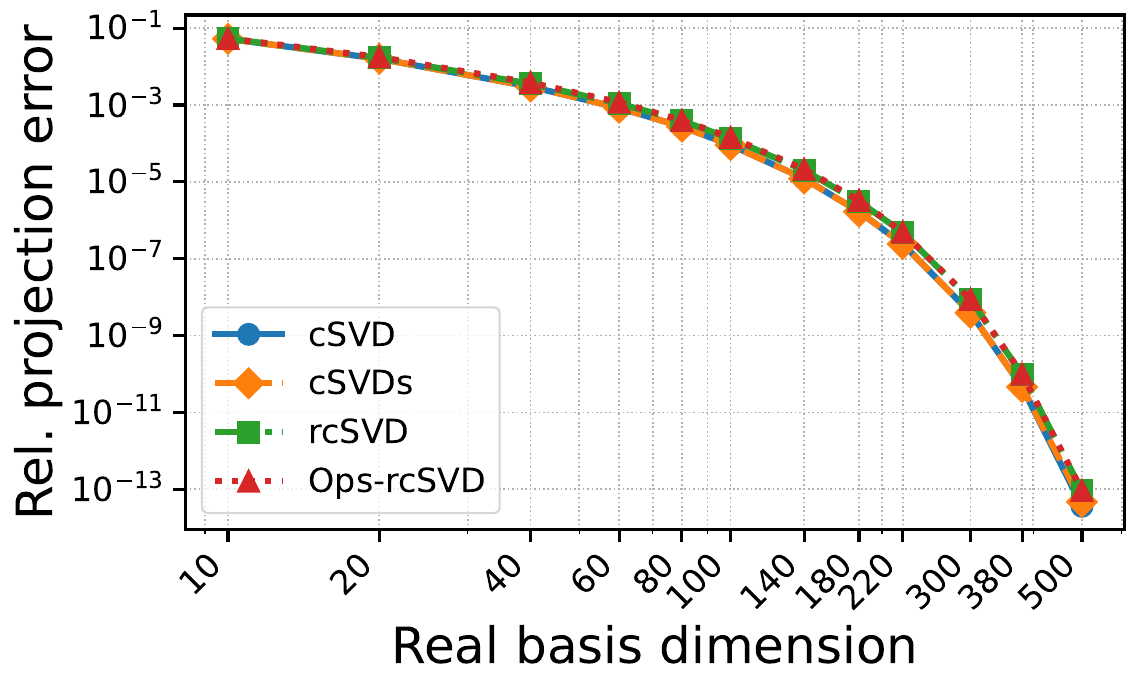}
        \par\smallskip
        \textnormal{(c) Relative projection error.}
    \end{minipage}
    \hfill
    \begin{minipage}[t]{0.48\textwidth}
    \vspace{0pt}

    \caption{Basis-construction results for the lossless
    LC resonator network.
    (a) peak memory consumption,
    (b) snapshot generation and basis-construction time,
    and (c) relative projection error.
    Randomized errors show the mean;
    runtimes and memory consumption show the
    median over repeated realizations.}
    \label{fig:ops-lc}
    \end{minipage}

\end{figure}

\subsection{Sketched reduced-order models}
\label{sec:numexp-sketched-rom}

We study sketched reduced-Hamiltonian assembly for a fixed
ortho-symplectic basis. The same basis is used for exact and sketched ROMs, so
their differences arise solely from reduced-Hamiltonian assembly. We report the FOM-relative state error
\begin{equation*}
e_{\mathrm{state}}
=
\frac{1}{
|\mathcal P_{\mathrm{test}}|(n_t+1)
}
\sum_{\mu\in\mathcal P_{\mathrm{test}}}
\sum_{j=0}^{n_t}
\frac{
\|x(t_j;\mu)-Vx_r(t_j;\mu)\|_2
}{
\|x(t_j;\mu)\|_2
}.
\end{equation*}
The exact-ROM value is the baseline obtained using the fixed basis and the
exactly assembled reduced Hamiltonian.

For each CountSketch realization, we also compute $\delta_{\max}
=
\max_{\mu\in\mathcal P_{\mathrm{test}}}
\delta(\mu),$ and the spectral condition number $\kappa_H
=
\kappa_2(\Theta_HV).$
For each sketch dimension, five independent CountSketch realizations are used;
accuracy and conditioning curves show the mean and standard deviation.

Assembly timing starts once $V$ is available and ends when all
reduced components have been formed. Exact timings include full-order
operator actions, projection, and symmetrization; sketched timings include CountSketch generation and application, formation of
the reduced matrices $G=(\Theta_HV)^\top(\Theta_HV)$ and $
M_q=(\Theta_HV)^\top(\Theta_HH_qV),$
least-squares recovery, and symmetrization. Each assembly is repeated $20$
times. For each sketch realization its median runtime is reported, because the mean is more sensitive to outliers.

\subsubsection{Linear Schr\"odinger equation}
\label{sec:numexp-sketched-schrodinger}

The large-scale Schr\"odinger experiment uses $
t\in[0,1]$ and $N_{\xi_1}=N_{\xi_2}=1000,$ which yields $N=10^6, 2N=2\times10^6.$ The six training trajectories contain approximately $6006$ snapshots, so
storing the complete complex snapshot matrix would require about
$96\,\mathrm{GB}$, exceeding the available memory.

A single Ops-rcSVD realization is therefore used to construct and fix a basis
of real dimension $2n=70.$ The same basis is used for all exact and sketched ROMs. Ops-rcSVD uses
the oversampling parameter $p = 70$ and block size $16$. For the independent Hamiltonian
CountSketch, $k_H/2n \in
\{2,4,8,16,32,64,128\},$ so $k_H\in\{140,280,560,1120,2240,4480,8960\}$. The results are presented in \Cref{fig:schrodinger_sketch_results}.

\begin{figure}[t]
    \centering

    \begin{subfigure}[t]{0.48\textwidth}
        \centering
        \includegraphics[width=\linewidth]
        {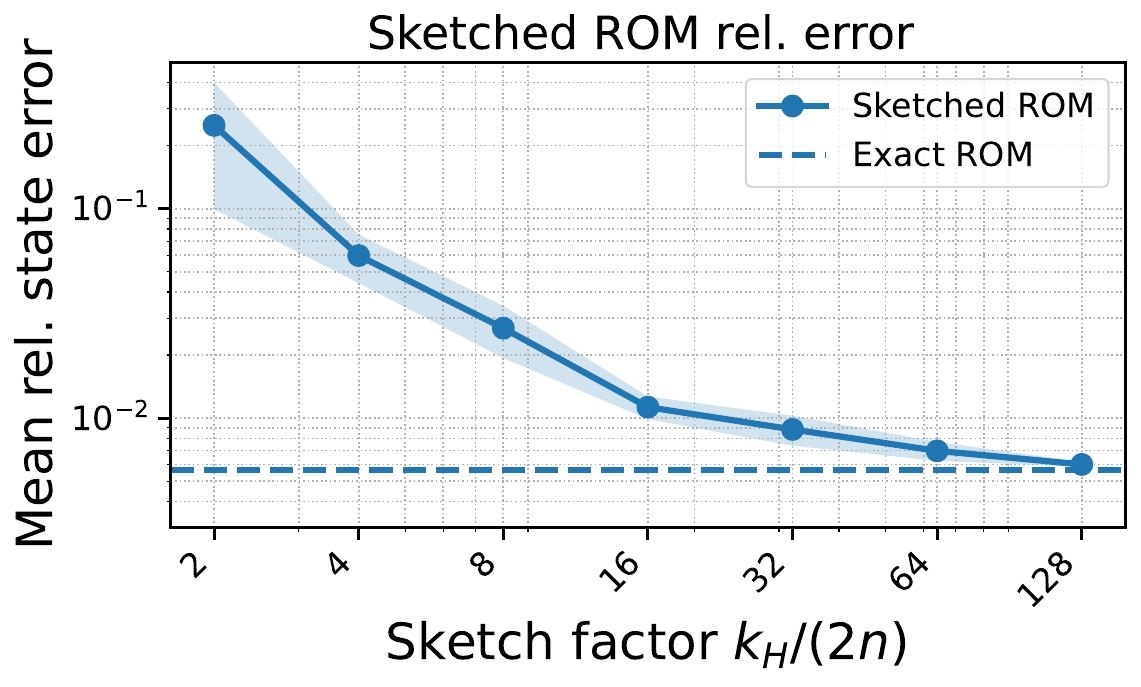}
        \caption{Relative state error.}
        \label{fig:schrodinger_sketch_results:state}
    \end{subfigure}
    \hfill
    \begin{subfigure}[t]{0.48\textwidth}
        \centering
        \includegraphics[width=\linewidth]
        {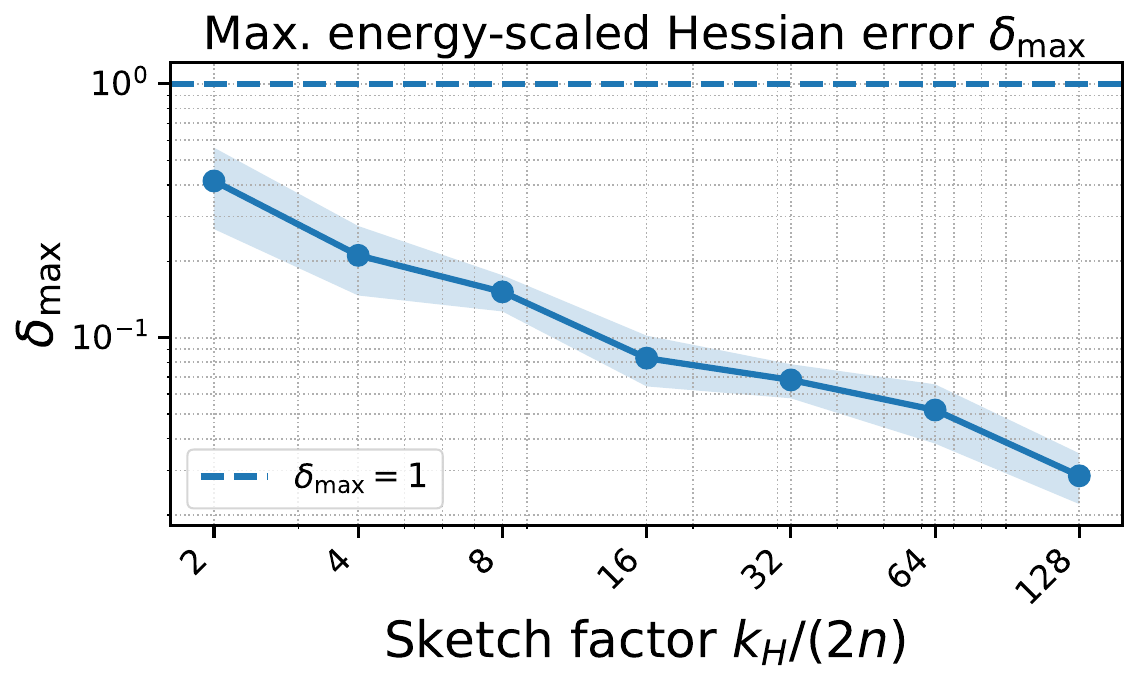}
        \caption{Maximum energy-scaled error $\delta_{\max}$.}
        \label{fig:schrodinger_sketch_results:delta}
    \end{subfigure}

    \vspace{0.5em}

    \begin{subfigure}[t]{0.48\textwidth}
        \centering
        \includegraphics[width=\linewidth]
        {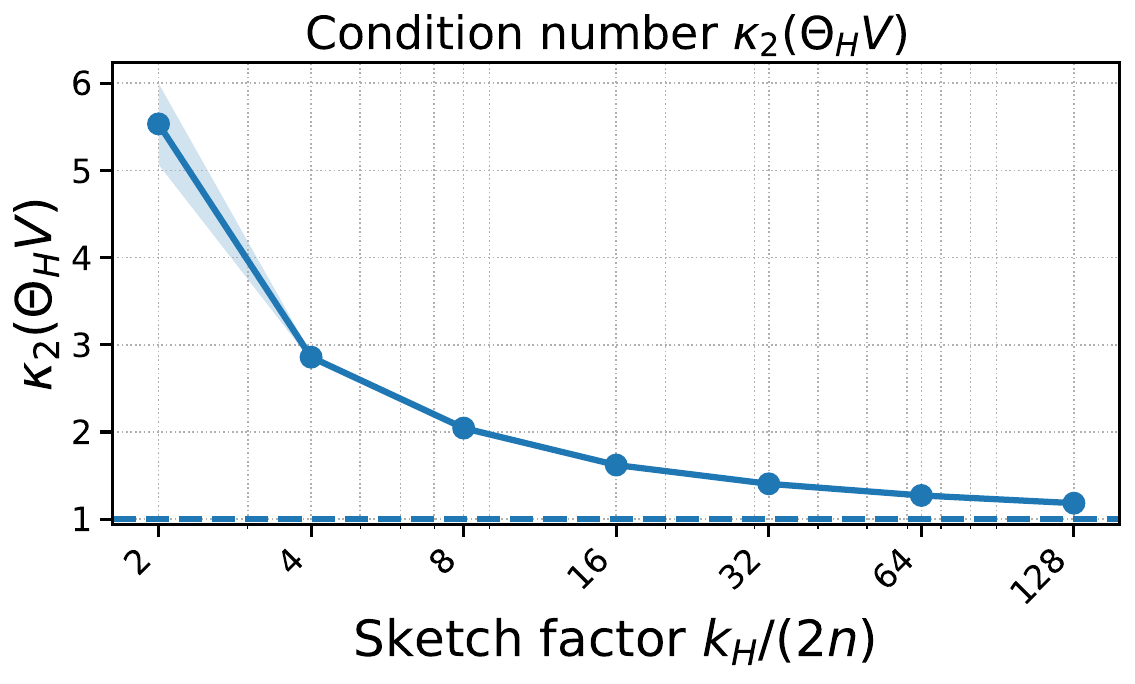}
        \caption{$\kappa_2(\Theta_H V)$.}
        \label{fig:schrodinger_sketch_results:condition}
    \end{subfigure}
    \hfill
    \begin{subfigure}[t]{0.48\textwidth}
        \centering
        \includegraphics[width=\linewidth]
        {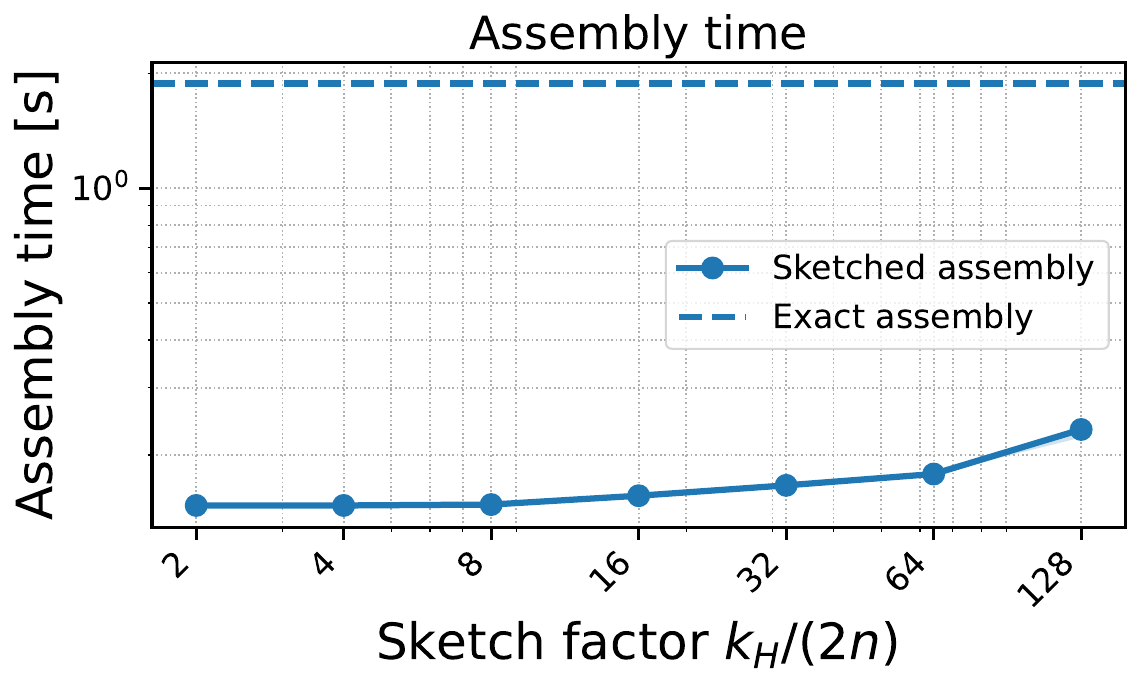}
        \caption{Reduced-Hamiltonian assembly time.}
        \label{fig:schrodinger_sketch_results:time}
    \end{subfigure}

    \caption{CountSketch results for the linear Schr\"odinger problem with
    $2n=70$.}
    \label{fig:schrodinger_sketch_results}
\end{figure}

As $k_H$ increases, the sketched-ROM state error approaches the exact-ROM
baseline. For the smallest sketch the error is of order $10^{-1}$, while
larger sketches reduce it steadily; at the largest dimension it is only
slightly above the exact-ROM level. The decreasing spread across realizations
also indicates reduced sensitivity to the particular CountSketch.

Simultaneously, $\delta_{\max}$ decreases from approximately
$4\times10^{-1}$ to about $3\times10^{-2}$, remaining below the
analysis threshold $\delta_{\max}=1$, while
$\kappa_2(\Theta_HV)$ decreases from roughly $5.5$ towards $1$.
As $\delta_{\max}$ and $\kappa_2(\Theta_HV)$ decrease, the state error
approaches the exact-ROM baseline. 

The sketched assembly remains substantially cheaper than exact projection.
Small and moderate sketches require below $0.2\,\mathrm{s}$, and the largest
only slightly above $0.2\,\mathrm{s}$, compared with approximately
$1.9\,\mathrm{s}$ for exact assembly. Thus moderate sketch dimensions
provide a favorable accuracy--cost tradeoff, while the largest sketches
nearly recover the exact-ROM baseline at substantially lower assembly cost.

\subsubsection{Lossless LC resonator network}
\label{sec:numexp-sketched-lc}

The large-scale LC experiment uses $t\in[0,120]$ and $N_{\xi_1}=N_{\xi_2}=1000,$ which yields $
N=10^6, 
2N=2\times10^6.$ The approximately $72006$ complex training snapshots would require about
$1.15\,\mathrm{TB}$ if stored explicitly. We fix one Ops-rcSVD basis of real dimension $2n=40,$ constructed with the oversampling parameter $p = 40$ and block size $16$. The independent
Hamiltonian CountSketch uses $k_H/2n \in
\{2,4,8,16,32,64,128,256,512,1024,2048,4096\},$ corresponding to $k_H=80,\ldots,163840$.

The reduced components exploit
$H_{r,0}
=
I+
V_\phi^\top G_{\mathrm{fixed}}V_\phi,$ and $
H_{r,1}
=
V_\phi^\top G_{\mathrm{coupling}}V_\phi.$
In the sketched implementation the identity contribution is included directly
in the sketch accumulation and is recovered exactly in exact arithmetic when
$\Theta_HV$ has full column rank.

\begin{figure}[t]
    \centering

    \begin{subfigure}[t]{0.48\textwidth}
        \centering
        \includegraphics[width=\linewidth]
        {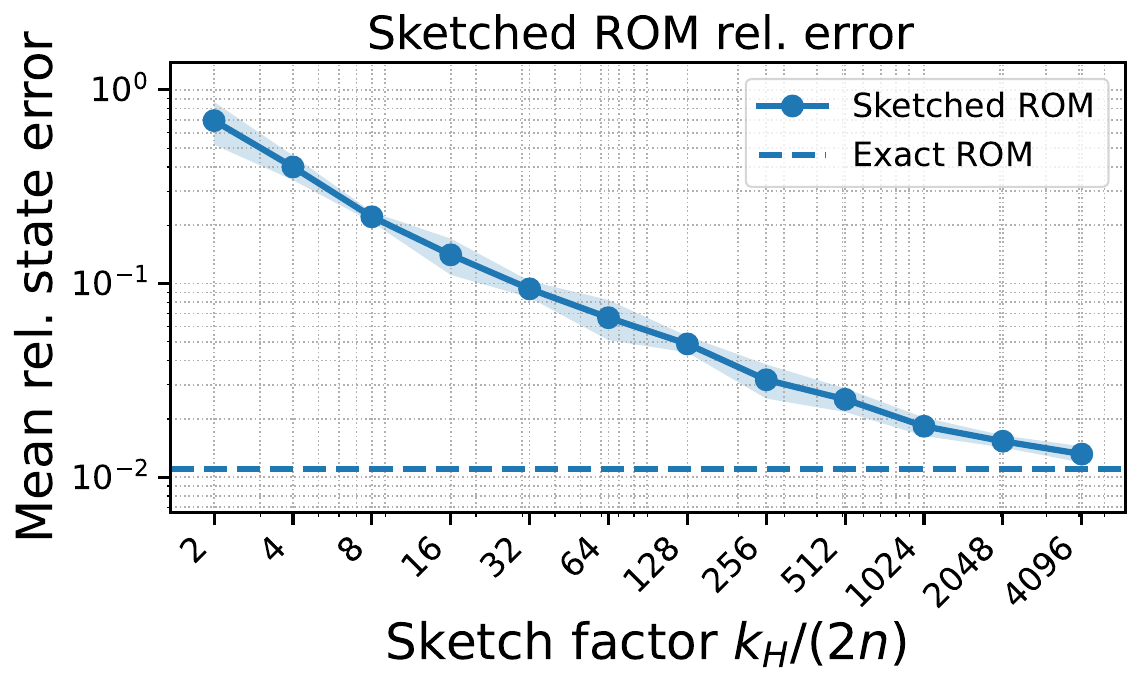}
        \caption{Relative state error.}
        \label{fig:lc-rom-error}
    \end{subfigure}
    \hfill
    \begin{subfigure}[t]{0.48\textwidth}
        \centering
        \includegraphics[width=\linewidth]
        {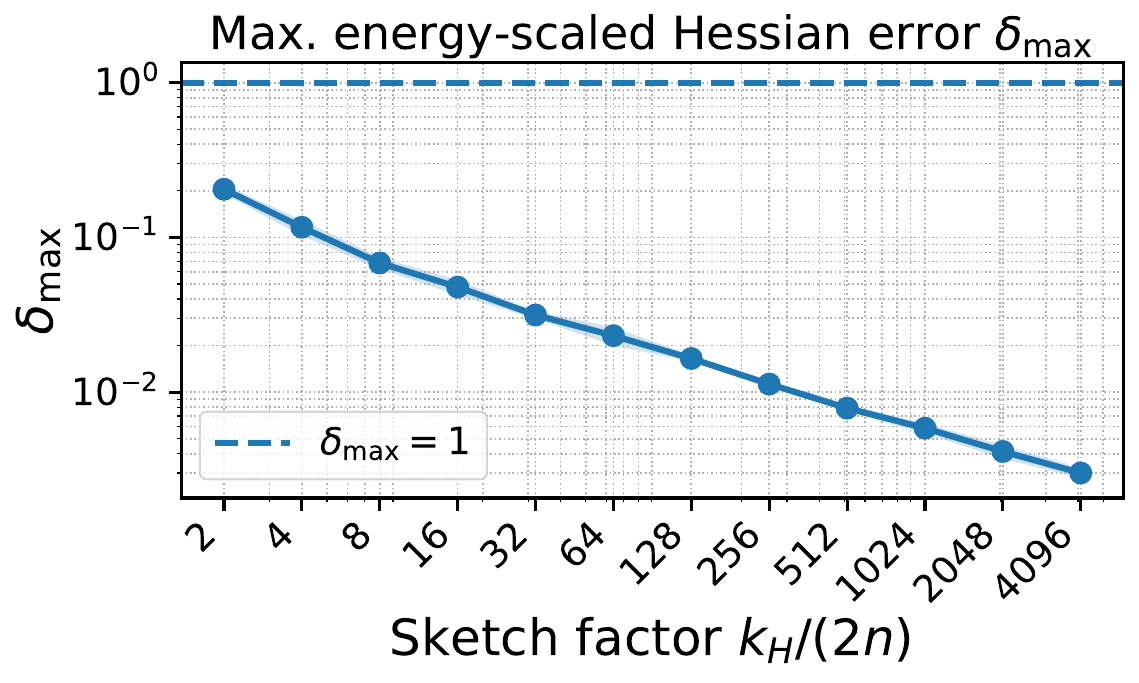}
        \caption{Maximum energy-scaled error $\delta_{\max}$.}
        \label{fig:lc-hessian-error}
    \end{subfigure}

    \vspace{0.4em}

    \begin{subfigure}[t]{0.48\textwidth}
        \centering
        \includegraphics[width=\linewidth]
        {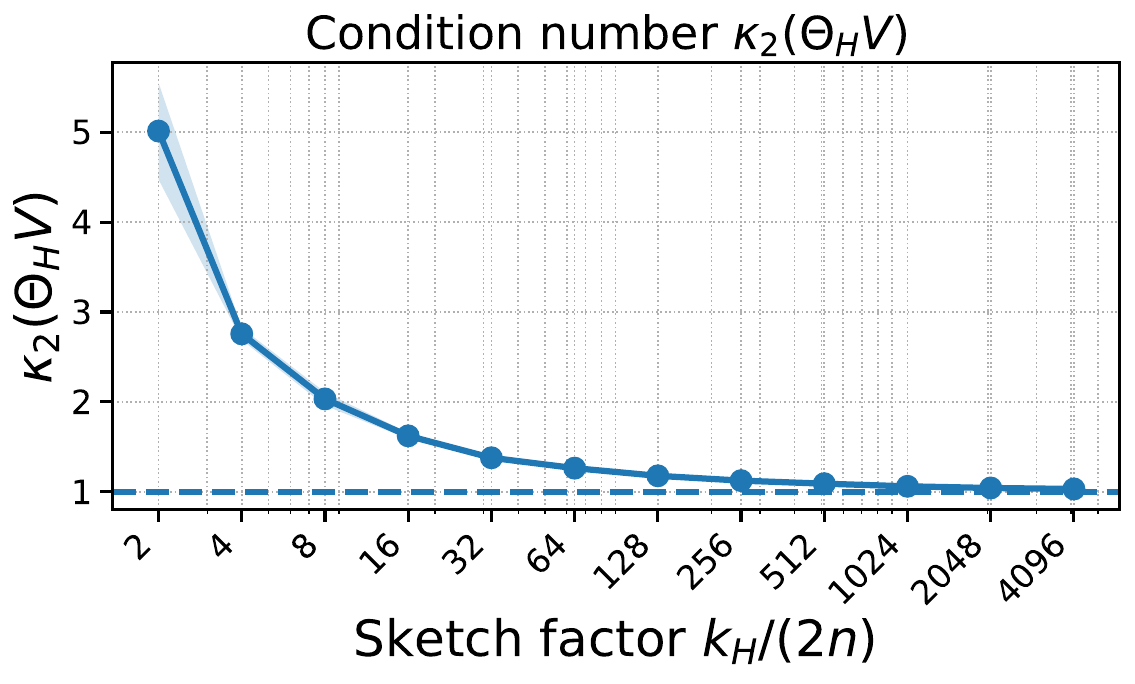}
        \caption{$\kappa_2(\Theta_H V)$.}
        \label{fig:lc-condition}
    \end{subfigure}
    \hfill
    \begin{subfigure}[t]{0.48\textwidth}
        \centering
        \includegraphics[width=\linewidth]
        {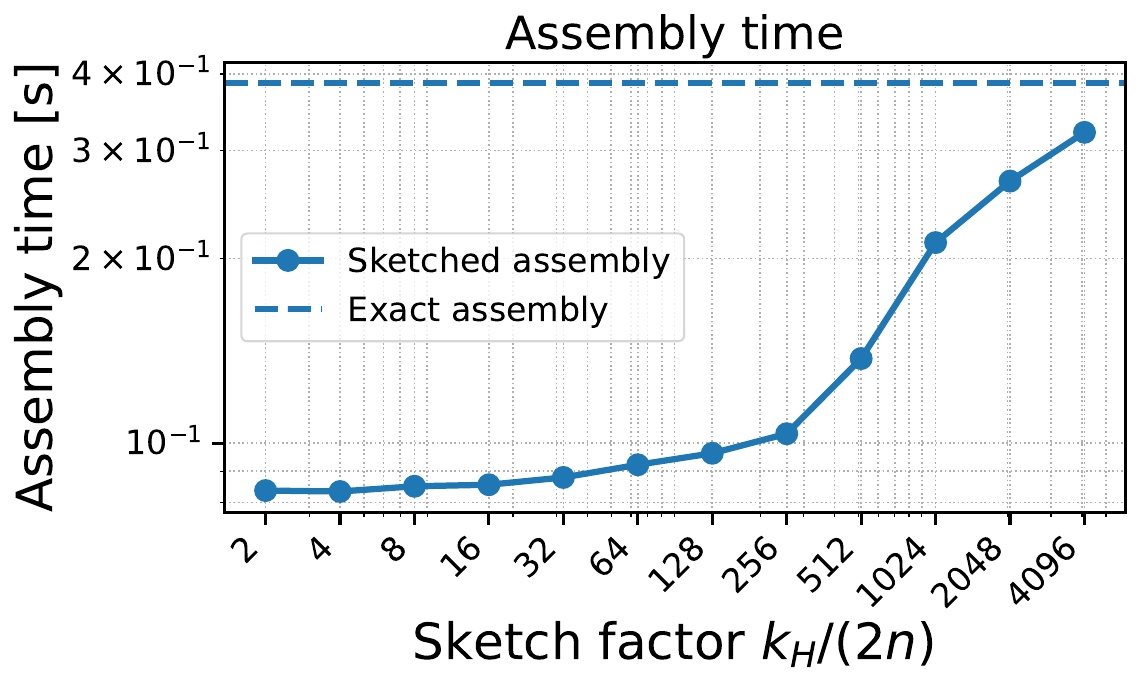}
        \caption{Reduced-Hamiltonian assembly time.}
        \label{fig:lc-assembly-time}
    \end{subfigure}

    \caption{CountSketch results for the lossless LC resonator network with
    $2n=40$.}
    \label{fig:lc-results}
\end{figure}

The LC experiment exhibits the same convergence trend but requires a much
wider range of sketch dimensions. As presented in \Cref{fig:lc-results} small sketches produce state errors of order
$10^{-1}$, well above the exact-ROM baseline slightly above $10^{-2}$.
The error decreases steadily with $k_H$, approaching the baseline only for
the largest sketches.

As in the Schr\"odinger experiment, increasing $k_H$ decreases both
$\delta_{\max}$ and $\kappa_2(\Theta_HV)$:
$\delta_{\max}$ decreases from slightly above $10^{-1}$ to order
$10^{-3}$, remaining below $1$, while
$\kappa_2(\Theta_HV)$ decreases from approximately $5$ towards $1$.
The simultaneous decrease of $\delta_{\max}$ and the state error is
consistent with the perturbation analysis and indicates that the assembly
error remains dynamically relevant over a larger range of sketch dimensions
than for the Schr\"odinger problem.

For small and moderate sketches, assembly remains below approximately
$0.1\,\mathrm{s}$. At large $k_H$, the costs of forming \(G\) and \(M_q\) grow with $k_H$ and the runtime exceeds $0.3\,\mathrm{s}$, compared with just under
$0.4\,\mathrm{s}$ for exact assembly. Thus CountSketch remains cheaper over
the tested range, but its advantage decreases as higher accuracy is demanded.
The LC experiment therefore illustrates the accuracy--cost tradeoff of the
one-sparse sketch particularly clearly.

\section{Conclusion}
\label{Sec:Conclusion}

We developed randomized techniques for reducing the offline cost of
structure-preserving model order reduction of Hamiltonian systems. The two
main targets were the construction of an ortho-symplectic reduced basis and
the assembly of the corresponding reduced Hamiltonian system.

For basis construction, we introduced Ops-rcSVD, a one-pass randomized
variant of the complex-SVD-based symplectic basis construction. It combines
randomized range finding with an independent left sketch and least-squares
coefficient recovery, processing the snapshot ensemble in a single streaming
pass and avoiding the second pass required by conventional rcSVD. The analysis separates the randomized
range-finding residual from the additional error introduced by the recovery of the projected coordinates.

For quadratic Hamiltonians, we introduced a sketched reduced-Hamiltonian
assembly based on sketching the reduced basis and Hamiltonian operator
actions, followed by least-squares reconstruction and explicit
symmetrization. For separable parameter dependence, all high-dimensional operator
and sketching operations are confined to the offline stage, and online
assembly uses only precomputed reduced components.

The analysis provides coefficient-recovery and snapshot projection bounds for
Ops-rcSVD, together with approximate-matrix-multiplication estimates for the
sketches considered. For the reduced-Hamiltonian assembly, subspace-embedding conditions on the
basis and operator-residual spaces yield uniform-in-parameter perturbation bounds and high-probability trajectory-error guarantees. The energy-scaled perturbation
$\delta(\mu)$ measures the relative reduced-Hamiltonian error;
$\delta(\mu)<1$ guarantees positive definiteness of the sketched reduced
Hessian whenever the exact one is positive definite. The Hamiltonian
structure further yields an energy-norm trajectory-error estimate with at
most linear growth in time, including a uniform-in-parameter version under
appropriate bounds.

Numerical experiments on a linear Schr\"odinger equation and a lossless LC
resonator network support these results. In the basis-construction study,
cSVD, cSVDs, rcSVD, and Ops-rcSVD produced reduced spaces of essentially the
same approximation quality over the tested dimensions. Ops-rcSVD consistently
gave the lowest runtime, while the close agreement with rcSVD indicates that
the one-pass least-squares recovery introduced no visible deterioration of
the reduced space. Its advantage stems from eliminating the second generation
and processing of the training trajectories.

A second study used $1000\times1000$ discretizations, corresponding to
$N=10^6$ spatial degrees of freedom and state dimension
$2N=2\cdot10^6$. Storing the complete complex training snapshot matrices
would require approximately $96\,\mathrm{GB}$ for the Schr\"odinger
problem and $1.15\,\mathrm{TB}$ for the LC network, placing these
experiments beyond the regime in which conventional full-snapshot cSVD is
practical.

For reduced-Hamiltonian assembly, an independent CountSketch $\Theta_H$
was applied to the basis and operator components. Increasing the sketch
dimension improves the conditioning of $\Theta_HV$ and decreases $\delta_{\max}
=
\max_{\mu\in\mathcal P_{\mathrm{test}}}\delta(\mu).$ At the same time, the sketched-ROM state error approaches the exact-ROM
baseline, consistent with the reduced-Hamiltonian perturbation and trajectory
bounds. The Schr\"odinger problem reaches the fixed-basis accuracy for
comparatively moderate sketches, whereas the LC network requires
substantially larger sketch dimensions. The LC experiment makes this accuracy--cost tradeoff particularly visible: moderate sketches provide
substantial assembly savings, while very large sketches improve accuracy but
reduce the computational advantage over exact projection.
Together, the two randomized stages reduce repeated snapshot access and
high-dimensional reduced-operator assembly while preserving Hamiltonian
structure.
Future work includes extending the sketched reduced-Hamiltonian construction to nonlinear systems and systems with nonseparable parameter dependence through structure-preserving
hyper-reduction, developing sharper problem-dependent sketch-dimension
estimates for sparse and structured embeddings, and comparing sketching
families with different application-cost and embedding-dimension tradeoffs.

\section*{Acknowledgments}
This work is supported by Deutsche Forschungsgemeinschaft (DFG, German Research Foundation) Project No. 566234291, and under under Germany´s Excellence Strategy – EXC 2075 – 390740016. We acknowledge the support by the Stuttgart Center for Simulation Science (SimTech).

\section*{Conflict of Interest}
The authors declare that they have no conflict of interest.

\section*{AI declaration}
The authors acknowledge the use of ChatGPT (OpenAI, GPT-5.6 Sol) to assist with code generation, debugging, language editing, rephrasing, manuscript organization, and consistency checks of the numerical implementation and manuscript. All AI-generated code and suggestions were critically reviewed, revised where necessary, and verified by the authors against the mathematical formulation and numerical results. The authors take full responsibility for the final content, implementation, scientific arguments, and conclusions of the work, which were not altered by the use of AI assistance.

\section*{Appendix}

\appendix

\section{Approximate matrix multiplication for the SRFT}
\label{app:srft_amm}

This appendix establishes the SRFT approximate-matrix-multiplication estimate
used in the analysis of the one-pass coefficient recovery.

\begin{proposition}[SRFT approximate matrix multiplication for the
Ops-rcSVD residual]
\label{prop:srft_amm}
Let $Q\in\mathbb C^{N\times\ell}$ satisfy $Q^*Q=I_\ell$, and let
$R_s\in\mathbb C^{N\times n_s}$ satisfy $Q^*R_s=0.$ Consider the SRFT matrix 
\[
\Theta
=
\sqrt{\frac{N}{k}}\,SFD
\in\mathbb C^{k\times N},
\]
where $F\in\mathbb C^{N\times N}$ is the unitary matrix of the discrete Fourier
transform, $D=\operatorname{diag}(d_1,\ldots,d_N)$ has independent
entries uniformly distributed on the complex unit circle, and
$S\in\mathbb R^{k\times N}$ selects $k$ coordinates uniformly without
replacement, independently of $D$.

Let $0<\delta<1$, define $\gamma_{\mathrm{SRFT}} := \frac{2(u+1)}{\sqrt{k}} \left( \sqrt{\ell}+L_\delta \right), L_\delta:=\sqrt{8\ln(N/\delta)},$
and let $0\le u\le\frac{\sqrt{k}}{1+L_\delta}.$
Then, with probability at least $1-e^{-u^2/4}-2\delta,$
we have
\begin{equation}
\label{eq:srft_amm_special}
\|Q^*(\Theta^*\Theta-I_N)R_s\|_F
\le
\gamma_{\mathrm{SRFT}}\|R_s\|_F.
\end{equation}
\end{proposition}

The proof follows the two ingredients used in
\cite[Lemma~4.11]{BoutsidisGittens2013}: randomized norm equalization and
uniform-sampling matrix multiplication.

\begin{lemma}[SRFT equalization of column norms]
\label{lem:srft_equalization}
Let $A\in\mathbb C^{m\times N}$, let
$F\in\mathbb C^{N\times N}$ be the unitary discrete Fourier transform,
and let $D=\operatorname{diag}(d_1,\ldots,d_N),$ where the $d_j$ are independent and uniformly distributed on the complex
unit circle. Then, for every $t\ge0$,
\begin{equation}
\label{eq:srft_equalization}
\mathbb P\left[
\max_{j=1,\ldots,N}
\|(AD^*F^*)_{(:,j)}\|_2
\le
\frac{1}{\sqrt N}
\left(
\|A\|_F+t\|A\|_2
\right)
\right]
\ge
1-Ne^{-t^2/8}.
\end{equation}
\end{lemma}

\begin{proof}
We write each random phase as $d_j=\varepsilon_j\zeta_j,$ where $\varepsilon_j\in\{-1,+1\}$ are independent Rademacher variables
and $\zeta_j$ are independent and uniformly distributed on a fixed
semicircle, independently of the $\varepsilon_j$. Then $d_j$ is
uniformly distributed on the complex unit circle. Next, we set $D=EZ, E=\operatorname{diag}(\varepsilon_1,\ldots,\varepsilon_N), Z=\operatorname{diag}(\zeta_1,\ldots,\zeta_N),$ and fix an arbitrary realization of $Z$.

We set $f_j:=F^*e_j$. The $j$-th column of $AD^*F^*$ can be written as
\[
AD^*F^*e_j
=
AZ^*Ef_j
=
AZ^*\operatorname{diag}(f_j)\varepsilon,
\]
where
$\varepsilon=(\varepsilon_1,\ldots,\varepsilon_N)^\top$.
Hence
\[
g_j(\varepsilon)
:=
\|AZ^*\operatorname{diag}(f_j)\varepsilon\|_2
\]
is a convex function of the Rademacher vector $\varepsilon$. Since every
entry of $f_j$ has modulus $N^{-1/2}$,
\[
|g_j(x)-g_j(y)|
\le
\frac{\|A\|_2}{\sqrt N}\|x-y\|_2.
\]
The convex-Lipschitz concentration inequality for Rademacher vectors used in
the proof of \cite[Lemma~4.6]{BoutsidisGittens2013} therefore yields
\[
\mathbb P_\varepsilon\left[
g_j(\varepsilon)
\ge
\mathbb E_\varepsilon[g_j(\varepsilon)]
+
\frac{t}{\sqrt N}\|A\|_2
\right]
\le
e^{-t^2/8}.
\]
Furthermore,
\begin{equation*}
\mathbb E_\varepsilon[g_j(\varepsilon)^2]
=
\operatorname{tr}\!\left(
AZ^*\operatorname{diag}(f_j)
\mathbb E[\varepsilon\varepsilon^\top]
\operatorname{diag}(f_j)^*
ZA^*
\right)
=
\frac{1}{N}\operatorname{tr}(AA^*)
=
\frac{1}{N}\|A\|_F^2,
\end{equation*}
and Jensen's inequality gives
\[
\mathbb E_\varepsilon[g_j(\varepsilon)]
\le
\frac{\|A\|_F}{\sqrt N}.
\]
The resulting estimate does not depend on the chosen realization of $Z$
and therefore remains valid after averaging over $Z$. Finally, a union
bound over $j=1,\ldots,N$ gives
\[
\mathbb P\left[
\max_j
\|(AD^*F^*)_{(:,j)}\|_2
>
\frac{\|A\|_F+t\|A\|_2}{\sqrt N}
\right]
\le
Ne^{-t^2/8},
\]
which proves \eqref{eq:srft_equalization}.
\end{proof}

Lemma~\ref{lem:srft_equalization} is the complex Fourier counterpart of the
SRHT equalization result \cite[Lemma~4.6]{BoutsidisGittens2013}.
Lemma~\ref{lem:complex_uniform_amm} is the complex extension of the uniform
sampling result \cite[Lemma~4.13]{BoutsidisGittens2013}.
\begin{lemma}[Complex matrix multiplication by uniform sampling]
\label{lem:complex_uniform_amm}
Let $X\in\mathbb C^{m\times N},$ and $Y\in\mathbb C^{N\times p}$
and let $k\le N$. We select $k$ indices uniformly without replacement and
form $\widehat X\in\mathbb C^{m\times k},$ and $\widehat Y\in\mathbb C^{k\times p}$ from the corresponding columns of $X$ and rows of $Y$, respectively,
with both sampled factors scaled by $\sqrt{N/k}$. Let
\[
\sigma^2
\ge
\frac{4N}{k}
\sum_{i=1}^N
\|X_{(:,i)}\|_2^2
\|Y_{(i,:)}\|_2^2
\]
and
\[
M
\ge
\frac{2N}{k}
\max_{i=1,\ldots,N}
\|X_{(:,i)}\|_2
\|Y_{(i,:)}\|_2.
\]
Then, for every $0\le\tau\le{\sigma^2}/{M},$ we have
\begin{equation}
\label{eq:complex_uniform_amm}
\mathbb P\left[
\|\widehat X\widehat Y-XY\|_F
\ge
\sigma+\tau
\right]
\le
\exp\!\left(
-\frac{\tau^2}{4\sigma^2}
\right).
\end{equation}
\end{lemma}

\begin{proof}
We identify
$\mathbb C^{m\times p}$ with $\mathbb R^{2mp}$ by separating real and
imaginary parts. Under this identification, the Frobenius norm is precisely
the Euclidean norm of the corresponding real vector.

The sampled rank-one products
\[
\frac{N}{k}
X_{(:,i)}Y_{(i,:)}
\]
can therefore be regarded as vectors in the real normed space
$\mathbb R^{2mp}$. Applying the vector Bernstein estimate for sampling
without replacement used in
\cite[Lemma~4.12]{BoutsidisGittens2013} gives
\eqref{eq:complex_uniform_amm} with the same bounds on $\sigma^2$ and
$M$.
\end{proof}

We now combine these two lemmas in exactly the two steps needed for
Proposition~\ref{prop:srft_amm}.

\begin{proof}[Proof of Proposition~\ref{prop:srft_amm}]
We set $A:=R_s^*, B:=Q,$ and define $X:=AD^*F^*,Y:=FDB.$ Since $F$ and $D$ are unitary,
\[
XY=AB=R_s^*Q=0.
\]
Moreover,
\[
A\Theta^*\Theta B
=
\frac{N}{k}\,XS^*SY.
\]
\paragraph{Step 1: Application of Lemma~\ref{lem:srft_equalization}.}
Applying Lemma~\ref{lem:srft_equalization} to $A$ with
$t=L_\delta$ gives
\[
\max_i\|X_{(:,i)}\|_2
\le
\frac{\|A\|_F+L_\delta\|A\|_2}{\sqrt N}
\]
with probability at least $1-\delta$. Applying the same lemma to $B^*$
gives
\[
\max_i\|Y_{(i,:)}\|_2
\le
\frac{\|B\|_F+L_\delta\|B\|_2}{\sqrt N}
\]
with probability at least $1-\delta$. Thus, by a union bound, both
estimates hold on an event $\mathcal E$ satisfying
\[
\mathbb P(\mathcal E)\ge1-2\delta.
\]
\paragraph{Step 2: Application of Lemma~\ref{lem:complex_uniform_amm}.}
We fix any realization of $D$ belonging to $\mathcal E$. Since the row
selection $S$ is independent of $D$,
Lemma~\ref{lem:complex_uniform_amm} applies to the resulting matrices
$X$ and $Y$.

We choose
\begin{equation}
\label{eq:srft_sigma}
\sigma^2
:=
\frac{4}{k}
\left(
\|B\|_F+L_\delta\|B\|_2
\right)^2
\|A\|_F^2.
\end{equation}
Using the bound on the row norms of $Y$ and
$\|X\|_F=\|A\|_F$, we obtain
\[
\frac{4N}{k}
\sum_{i=1}^N
\|X_{(:,i)}\|_2^2
\|Y_{(i,:)}\|_2^2
\le
\frac{4N}{k}
\max_i\|Y_{(i,:)}\|_2^2
\|X\|_F^2\le
\sigma^2.
\]
Next, we set
\begin{equation}
\label{eq:srft_M}
M
:=
\frac{2}{k}
\left(
\|A\|_F+L_\delta\|A\|_2
\right)
\left(
\|B\|_F+L_\delta\|B\|_2
\right).
\end{equation}
The two bounds defining $\mathcal E$ give
\[
M
\ge
\frac{2N}{k}
\max_i
\|X_{(:,i)}\|_2
\|Y_{(i,:)}\|_2.
\]
From \eqref{eq:srft_sigma}--\eqref{eq:srft_M},
\[
\frac{\sigma^2}{M}
=
\frac{
2\|A\|_F^2
\left(
\|B\|_F+L_\delta\|B\|_2
\right)
}{
\|A\|_F+L_\delta\|A\|_2
}.
\]
Since $\|A\|_2\le\|A\|_F$,
\[
\frac{\sigma^2}{M}
\ge
\frac{
2\|A\|_F
\left(
\|B\|_F+L_\delta\|B\|_2
\right)
}{
1+L_\delta
}
=
\frac{\sqrt{k}}{1+L_\delta}\,\sigma.
\]
Hence
\[
0\le u\le\frac{\sqrt{k}}{1+L_\delta}
\]
implies
\[
u\sigma\le\frac{\sigma^2}{M}.
\]
Using $\tau=u\sigma$ in
Lemma~\ref{lem:complex_uniform_amm} therefore gives, for every realization
of $D$ in $\mathcal E$,
\[
\mathbb P_S\left[
\|A\Theta^*\Theta B-AB\|_F
\ge
(u+1)\sigma
\right]
\le
e^{-u^2/4}.
\]
Together with $\mathbb P(\mathcal E)\ge1-2\delta$, this yields
\[
\mathbb P\left[
\|A\Theta^*\Theta B-AB\|_F
\le
\frac{2(u+1)}{\sqrt{k}}
\|A\|_F
\left(
\|B\|_F+L_\delta\|B\|_2
\right)
\right]
\ge
1-e^{-u^2/4}-2\delta.
\]
Finally, for $A=R_s^*$ and $B=Q$,
\[
\|A\|_F=\|R_s\|_F,
\qquad
\|Q\|_F=\sqrt{\ell},
\qquad
\|Q\|_2=1,
\qquad
AB=R_s^*Q=0.
\]
Since
\[
\|R_s^*\Theta^*\Theta Q\|_F
=
\|Q^*\Theta^*\Theta R_s\|_F
=
\|Q^*(\Theta^*\Theta-I_N)R_s\|_F,
\]
we obtain
\[
\|Q^*(\Theta^*\Theta-I_N)R_s\|_F
\le
\frac{2(u+1)}{\sqrt{k}}
\left(
\sqrt{\ell}+L_\delta
\right)
\|R_s\|_F,
\]
which is \eqref{eq:srft_amm_special}.
\end{proof}

The resulting constant
$\gamma_{\mathrm{SRFT}}
=2(u+1)(\sqrt{\ell}+L_\delta)/\sqrt{k}$
can therefore be inserted directly into the coefficient-recovery bound of
Proposition~\ref{prop:coeff_error_B_BLS}.


\begin{thebibliography}{45}
\providecommand{\natexlab}[1]{#1}
\providecommand{\url}[1]{\texttt{#1}}
\expandafter\ifx\csname urlstyle\endcsname\relax
  \providecommand{\doi}[1]{doi: #1}\else
  \providecommand{\doi}{doi: \begingroup \urlstyle{rm}\Url}\fi

\bibitem[Meyer and Offin(2017)]{Meyer2017}
Kenneth~R. Meyer and Daniel~C. Offin.
\newblock \emph{Introduction to Hamiltonian Dynamical Systems and the N-Body
  Problem}, volume~90 of \emph{Applied Mathematical Sciences}.
\newblock Springer Cham, 3rd edition, 2017.
\newblock \doi{10.1007/978-3-319-53691-0}.

\bibitem[Benner et~al.(2015)Benner, Gugercin, and Willcox]{Benner2015}
Peter Benner, Serkan Gugercin, and Karen Willcox.
\newblock A survey of projection-based model reduction methods for parametric
  dynamical systems.
\newblock \emph{SIAM Review}, 57\penalty0 (4):\penalty0 483--531, 2015.
\newblock \doi{10.1137/130932715}.

\bibitem[Peng and Mohseni(2016)]{Peng2016}
Liqian Peng and Kamran Mohseni.
\newblock Symplectic model reduction of {H}amiltonian systems.
\newblock \emph{SIAM Journal on Scientific Computing}, 38\penalty0
  (1):\penalty0 A1--A27, 2016.
\newblock \doi{10.1137/140978922}.

\bibitem[{Maboudi Afkham} and Hesthaven(2017)]{MaboudiAfkham2017}
Babak {Maboudi Afkham} and Jan~S. Hesthaven.
\newblock Structure preserving model reduction of parametric {H}amiltonian
  systems.
\newblock \emph{SIAM Journal on Scientific Computing}, 39\penalty0
  (6):\penalty0 A2616--A2644, 2017.
\newblock \doi{10.1137/17M1111991}.

\bibitem[Buchfink et~al.(2019)Buchfink, Bhatt, and Haasdonk]{Buchfink2019}
Patrick Buchfink, Ashish Bhatt, and Bernard Haasdonk.
\newblock Symplectic model order reduction with non-orthonormal bases.
\newblock 24\penalty0 (2):\penalty0 43, 2019.
\newblock ISSN 2297-8747.
\newblock \doi{10.3390/mca24020043}.

\bibitem[Fresca and Manzoni(2022)]{Fresca2022}
Stefania Fresca and Andrea Manzoni.
\newblock {POD-DL-ROM}: Enhancing deep learning-based reduced order models for
  nonlinear parametrized {PDEs} by proper orthogonal decomposition.
\newblock \emph{Computer Methods in Applied Mechanics and Engineering},
  388:\penalty0 114181, 2022.
\newblock \doi{10.1016/j.cma.2021.114181}.

\bibitem[Rettberg et~al.(2025)Rettberg, Kneifl, Herb, Buchfink, Fehr, and
  Haasdonk]{Rettberg2025}
Johannes Rettberg, Jonas Kneifl, Julius Herb, Patrick Buchfink, J{\"o}rg Fehr,
  and Bernard Haasdonk.
\newblock Data-driven identification of latent port-{Hamiltonian} systems.
\newblock \emph{Computational Science and Engineering}, 2:\penalty0 4, 2025.
\newblock \doi{10.1007/s44207-025-00007-2}.

\bibitem[Brantner and Kraus(2023)]{brantner2023symplectic}
Benedikt Brantner and Michael Kraus.
\newblock Symplectic autoencoders for model reduction of hamiltonian systems.
\newblock \emph{arXiv preprint arXiv:2312.10004}, 2023.
\newblock \doi{https://doi.org/10.48550/arXiv.2312.10004}.

\bibitem[Xu(2003)]{Xu2003}
Hongguo Xu.
\newblock An {SVD}-like matrix decomposition and its applications.
\newblock \emph{Linear Algebra and its Applications}, 368:\penalty0 1--24,
  2003.
\newblock \doi{10.1016/S0024-3795(03)00370-7}.

\bibitem[Pagliantini(2021)]{Pagliantini2021}
Cecilia Pagliantini.
\newblock Dynamical reduced basis methods for {H}amiltonian systems.
\newblock \emph{Numerische Mathematik}, 148\penalty0 (2):\penalty0 409--448,
  2021.
\newblock ISSN 0945-3245.
\newblock \doi{10.1007/s00211-021-01211-w}.

\bibitem[Hesthaven et~al.(2022)Hesthaven, Pagliantini, and
  Ripamonti]{HesthavenPagliantiniRipamonti2022}
Jan~S. Hesthaven, Cecilia Pagliantini, and Nicol{\`o} Ripamonti.
\newblock Rank-adaptive structure-preserving model order reduction of
  {{H}amiltonian} systems.
\newblock \emph{ESAIM: Mathematical Modelling and Numerical Analysis},
  56\penalty0 (2):\penalty0 617--650, 2022.
\newblock \doi{10.1051/m2an/2022013}.

\bibitem[Schleu\ss{} et~al.(2023)Schleu\ss{}, Smetana, and ter
  Maat]{Schleuss23}
Julia Schleu\ss{}, Kathrin Smetana, and Lukas ter Maat.
\newblock Randomized quasi-optimal local approximation spaces in time.
\newblock \emph{SIAM Journal on Scientific Computing}, 45\penalty0
  (3):\penalty0 A1066--A1096, 2023.
\newblock \doi{10.1137/22M1481002}.

\bibitem[Herkert et~al.(2024{\natexlab{a}})Herkert, Buchfink, and
  Haasdonk]{Herkert2023dict}
Robin Herkert, Patrick Buchfink, and Bernard Haasdonk.
\newblock Dictionary-based online-adaptive structure-preserving model order
  reduction for parametric {H}amiltonian systems.
\newblock \emph{Advances in Computational Mathematics}, 50\penalty0
  (1):\penalty0 12, 2024{\natexlab{a}}.
\newblock \doi{10.1007/s10444-023-10102-7}.

\bibitem[Sharma et~al.(2023)Sharma, Mu, Buchfink, Geelen, Glas, and
  Kramer]{Sharma2023QM}
Harsh Sharma, Hongliang Mu, Patrick Buchfink, Rudy Geelen, Silke Glas, and
  Boris Kramer.
\newblock Symplectic model reduction of {H}amiltonian systems using data-driven
  quadratic manifolds.
\newblock \emph{Computer Methods in Applied Mechanics and Engineering},
  417:\penalty0 116402, 2023.
\newblock \doi{10.1016/j.cma.2023.116402}.

\bibitem[Buchfink et~al.(2023)Buchfink, Glas, and Haasdonk]{Buchfink2023SMG}
Patrick Buchfink, Silke Glas, and Bernard Haasdonk.
\newblock Symplectic model reduction of {H}amiltonian systems on nonlinear
  manifolds and approximation with weakly symplectic autoencoder.
\newblock \emph{SIAM Journal on Scientific Computing}, 45\penalty0
  (2):\penalty0 A289--A311, 2023.
\newblock \doi{10.1137/21M1466657}.

\bibitem[Mahoney(2010)]{Mahoney2010}
Michael~W. Mahoney.
\newblock Randomized algorithms for matrices and data.
\newblock \emph{Foundations and Trends{\textregistered} in Machine Learning},
  3\penalty0 (2):\penalty0 123--224, 2010.
\newblock \doi{10.1561/2200000035}.

\bibitem[Woodruff(2014)]{Woodruff2014}
David~P. Woodruff.
\newblock Sketching as a tool for numerical linear algebra.
\newblock \emph{Foundations and Trends{\textregistered} in Theoretical Computer
  Science}, 10\penalty0 (1--2):\penalty0 1--157, 2014.
\newblock \doi{10.1561/0400000060}.

\bibitem[Murray et~al.(2023)Murray, Demmel, Mahoney, Erichson, Melnichenko,
  Malik, Grigori, Luszczek, Derezi{\'n}ski, Lopes, Liang, Luo, and
  Dongarra]{Murray2023}
Riley Murray, James Demmel, Michael~W. Mahoney, N.~Benjamin Erichson, Maksim
  Melnichenko, Osman~Asif Malik, Laura Grigori, Piotr Luszczek, Micha{\l}
  Derezi{\'n}ski, Miles~E. Lopes, Tianyu Liang, Hengrui Luo, and Jack Dongarra.
\newblock Randomized numerical linear algebra: A perspective on the field with
  an eye to software.
\newblock Technical Report UCB/EECS-2023-19, EECS Department, University of
  California, Berkeley, February 2023.
\newblock URL
  \url{https://www2.eecs.berkeley.edu/Pubs/TechRpts/2023/EECS-2023-19.html}.

\bibitem[Halko et~al.(2011)Halko, Martinsson, and Tropp]{Halko2011}
Nathan Halko, Per-Gunnar Martinsson, and Joel~A. Tropp.
\newblock Finding structure with randomness: Probabilistic algorithms for
  constructing approximate matrix decompositions.
\newblock \emph{SIAM Review}, 53\penalty0 (2):\penalty0 217--288, 2011.
\newblock \doi{10.1137/090771806}.

\bibitem[Gu(2015)]{Gu2015}
Ming Gu.
\newblock Subspace iteration randomization and singular value problems.
\newblock \emph{SIAM Journal on Scientific Computing}, 37\penalty0
  (3):\penalty0 A1139--A1173, 2015.
\newblock \doi{10.1137/130938700}.

\bibitem[Tropp(2011)]{Tropp2011Ana}
Joel~A. Tropp.
\newblock Improved analysis of the subsampled randomized {H}adamard transform.
\newblock \emph{Advances in Adaptive Data Analysis}, 3\penalty0
  (1--2):\penalty0 115--126, 2011.
\newblock \doi{10.1142/S1793536911000787}.

\bibitem[Demmel et~al.(2023)Demmel, Grigori, and Rusciano]{Demmel2023}
James Demmel, Laura Grigori, and Alexander Rusciano.
\newblock An improved analysis and unified perspective on deterministic and
  randomized low-rank matrix approximation.
\newblock \emph{SIAM Journal on Matrix Analysis and Applications}, 44\penalty0
  (2):\penalty0 559--591, 2023.
\newblock \doi{10.1137/21M1391316}.

\bibitem[Alla and Kutz(2019)]{Alla2019}
Alessandro Alla and J.~Nathan Kutz.
\newblock Randomized model order reduction.
\newblock \emph{Advances in Computational Mathematics}, 45:\penalty0
  1251--1271, 2019.
\newblock \doi{10.1007/s10444-018-09655-9}.

\bibitem[Bach et~al.(2019)Bach, Ceglia, Song, and Duddeck]{Bach2019}
Christian Bach, Daniele Ceglia, Lailong Song, and Fabian Duddeck.
\newblock Randomized low-rank approximation methods for projection-based model
  order reduction of large nonlinear dynamical problems.
\newblock \emph{International Journal for Numerical Methods in Engineering},
  118\penalty0 (4):\penalty0 209--241, 2019.
\newblock \doi{10.1002/nme.6009}.

\bibitem[Buhr and Smetana(2018)]{Buhr2018}
Andreas Buhr and Kathrin Smetana.
\newblock Randomized local model order reduction.
\newblock \emph{SIAM Journal on Scientific Computing}, 40\penalty0
  (4):\penalty0 A2120--A2151, 2018.
\newblock \doi{10.1137/17M1138480}.

\bibitem[Pelling and Sarradj(2026)]{PellingSarradj2026}
Art J.~R. Pelling and Ennes Sarradj.
\newblock Adaptive reduced order modelling of acoustic {LTI} systems with
  input--output dead time.
\newblock \emph{Mechanical Systems and Signal Processing}, 242:\penalty0
  113613, 2026.
\newblock \doi{10.1016/j.ymssp.2025.113613}.

\bibitem[Mir{\'o} et~al.(2026)Mir{\'o}, Eiximeno, Gasparino, Kutz, Rodriguez,
  and Lehmkuhl]{MiroEtAl2026}
Arnau Mir{\'o}, Benet Eiximeno, Lucas Gasparino, Nathan Kutz, Ivette Rodriguez,
  and Oriol Lehmkuhl.
\newblock Toward a {GPU}-enabled billionaire {SVD} in {pyLOM}.
\newblock \emph{Acta Mechanica}, 2026.
\newblock \doi{10.1007/s00707-025-04621-1}.

\bibitem[Herkert et~al.(2024{\natexlab{b}})Herkert, Buchfink, Haasdonk,
  Rettberg, and Fehr]{Herkert2023rand}
Robin Herkert, Patrick Buchfink, Bernard Haasdonk, Johannes Rettberg, and
  J{\"o}rg Fehr.
\newblock Randomized symplectic model order reduction for {Hamiltonian}
  systems.
\newblock In Ivan Lirkov and Svetozar Margenov, editors, \emph{Large-Scale
  Scientific Computations}, volume 13952 of \emph{Lecture Notes in Computer
  Science}, pages 99--107. Springer Cham, 2024{\natexlab{b}}.
\newblock \doi{10.1007/978-3-031-56208-2_9}.

\bibitem[Herkert et~al.(2026)Herkert, Buchfink, Haasdonk, Rettberg, and
  Fehr]{herkert2025error}
Robin Herkert, Patrick Buchfink, Bernard Haasdonk, Johannes Rettberg, and
  J{\"o}rg Fehr.
\newblock Error analysis of randomized symplectic model order reduction for
  {H}amiltonian systems.
\newblock \emph{Linear Algebra and its Applications}, 729:\penalty0 67--99,
  2026.
\newblock ISSN 0024-3795.
\newblock \doi{10.1016/j.laa.2025.09.022}.

\bibitem[Balabanov and Nouy(2019)]{Balabanov2019}
Oleg Balabanov and Anthony Nouy.
\newblock Randomized linear algebra for model reduction. {P}art {I}: Galerkin
  methods and error estimation.
\newblock \emph{Advances in Computational Mathematics}, 45\penalty0
  (5):\penalty0 2969--3019, 2019.
\newblock ISSN 1572-9044.
\newblock \doi{10.1007/s10444-019-09725-6}.

\bibitem[Balabanov and Nouy(2021)]{Balabanov2021}
Oleg Balabanov and Anthony Nouy.
\newblock Randomized linear algebra for model reduction--part {II}: minimal
  residual methods and dictionary-based approximation.
\newblock \emph{Advances in Computational Mathematics}, 47\penalty0
  (2):\penalty0 26, 2021.
\newblock ISSN 1572-9044.
\newblock \doi{10.1007/s10444-020-09836-5}.

\bibitem[Balabanov(2019)]{Balabanov2019a}
Oleg Balabanov.
\newblock \emph{Randomized Linear Algebra for Model Order Reduction}.
\newblock PhD thesis, \'Ecole Centrale de Nantes and Universitat Polit\`ecnica
  de Catalunya, 2019.

\bibitem[Smetana et~al.(2019)Smetana, Zahm, and Patera]{Smetana19}
Kathrin Smetana, Olivier Zahm, and Anthony~T. Patera.
\newblock Randomized residual-based error estimators for parametrized
  equations.
\newblock \emph{SIAM Journal on Scientific Computing}, 41\penalty0
  (2):\penalty0 A900--A926, 2019.
\newblock \doi{10.1137/18M120364X}.

\bibitem[Benner et~al.(2021)Benner, Grivet-Talocia, Quarteroni, Rozza,
  Schilders, and Silveira]{Benner2021}
Peter Benner, Stefano Grivet-Talocia, Alfio Quarteroni, Gianluigi Rozza, Wil
  Schilders, and Lu{\'i}s~Miguel Silveira, editors.
\newblock \emph{Model Order Reduction, Volume 2: Snapshot-Based Methods and
  Algorithms}.
\newblock De Gruyter, 2021.
\newblock \doi{10.1515/9783110671490}.

\bibitem[Benner et~al.(2017)Benner, Cohen, Ohlberger, and Willcox]{Benner2017}
Peter Benner, Albert Cohen, Mario Ohlberger, and Karen Willcox, editors.
\newblock \emph{{M}odel {R}eduction and {A}pproximation: {T}heory and
  {A}lgorithms}, volume~15 of \emph{Computational Science \& Engineering}.
\newblock Society for Industrial and Applied Mathematics, Philadelphia, PA,
  2017.
\newblock \doi{10.1137/1.9781611974829}.

\bibitem[Hairer et~al.(2006)Hairer, Lubich, and Wanner]{Hairer2006GNI}
Ernst Hairer, Christian Lubich, and Gerhard Wanner.
\newblock \emph{Geometric Numerical Integration: Structure-Preserving
  Algorithms for Ordinary Differential Equations}, volume~31 of \emph{Springer
  Series in Computational Mathematics}.
\newblock Springer-Verlag, Berlin, 2nd edition, 2006.
\newblock \doi{10.1007/3-540-30666-8}.

\bibitem[Palitta et~al.(2025)Palitta, Schweitzer, and
  Simoncini]{PalittaSchweitzerSimoncini2025}
Davide Palitta, Marcel Schweitzer, and Valeria Simoncini.
\newblock Sketched and truncated polynomial {K}rylov methods: Evaluation of
  matrix functions.
\newblock \emph{Numerical Linear Algebra with Applications}, 32\penalty0
  (1):\penalty0 e2596, 2025.
\newblock \doi{10.1002/nla.2596}.

\bibitem[Brand(2006)]{Brand2006}
Matthew Brand.
\newblock Fast low-rank modifications of the thin singular value decomposition.
\newblock \emph{Linear Algebra and its Applications}, 415\penalty0
  (1):\penalty0 20--30, 2006.
\newblock \doi{10.1016/j.laa.2005.07.021}.

\bibitem[Fareed et~al.(2018)Fareed, Singler, Zhang, and Shen]{Fareed2018}
Hiba Fareed, John~R. Singler, Yangwen Zhang, and Jiguang Shen.
\newblock Incremental proper orthogonal decomposition for pde simulation data.
\newblock \emph{Computers \& Mathematics with Applications}, 75\penalty0
  (6):\penalty0 1942--1960, 2018.
\newblock \doi{10.1016/j.camwa.2017.09.012}.

\bibitem[Rettberg et~al.(2024)Rettberg, Wittwar, Buchfink, Herkert, Fehr, and
  Haasdonk]{RettbergEtAl2024}
Johannes Rettberg, Dominik Wittwar, Patrick Buchfink, Robin Herkert, J{\"o}rg
  Fehr, and Bernard Haasdonk.
\newblock Improved a posteriori error bounds for reduced port-{H}amiltonian
  systems.
\newblock \emph{Advances in Computational Mathematics}, 50:\penalty0 100, 2024.
\newblock \doi{10.1007/s10444-024-10195-8}.

\bibitem[Karas{\"o}zen and Uzunca(2018)]{KarasozenUzunca2018}
B{\"u}lent Karas{\"o}zen and Murat Uzunca.
\newblock Energy preserving model order reduction of the nonlinear
  {S}chr{\"o}dinger equation.
\newblock \emph{Advances in Computational Mathematics}, 44\penalty0
  (6):\penalty0 1769--1796, 2018.
\newblock \doi{10.1007/s10444-018-9593-9}.

\bibitem[Polyuga and van~der Schaft(2010)]{PolyugaVanderSchaft2010}
Rostyslav~V. Polyuga and Arjan~J. van~der Schaft.
\newblock Structure preserving model reduction of port-{H}amiltonian systems by
  moment matching at infinity.
\newblock \emph{Automatica}, 46\penalty0 (4):\penalty0 665--672, 2010.
\newblock \doi{10.1016/j.automatica.2010.01.018}.

\bibitem[Gugercin et~al.(2012)Gugercin, Polyuga, Beattie, and van~der
  Schaft]{GugercinEtAl2012}
Serkan Gugercin, Rostyslav~V. Polyuga, Christopher Beattie, and Arjan~J.
  van~der Schaft.
\newblock Structure-preserving tangential interpolation for model reduction of
  port-{H}amiltonian systems.
\newblock \emph{Automatica}, 48\penalty0 (9):\penalty0 1963--1974, 2012.
\newblock \doi{10.1016/j.automatica.2012.05.052}.

\bibitem[Chaturantabut et~al.(2016)Chaturantabut, Beattie, and
  Gugercin]{ChaturantabutEtAl2016}
Saifon Chaturantabut, Christopher Beattie, and Serkan Gugercin.
\newblock Structure-preserving model reduction for nonlinear port-{H}amiltonian
  systems.
\newblock \emph{SIAM Journal on Scientific Computing}, 38\penalty0
  (5):\penalty0 B837--B865, 2016.
\newblock \doi{10.1137/15M1055085}.

\bibitem[Boutsidis and Gittens(2013)]{BoutsidisGittens2013}
Christos Boutsidis and Alex Gittens.
\newblock Improved matrix algorithms via the subsampled randomized {H}adamard
  transform.
\newblock \emph{SIAM Journal on Matrix Analysis and Applications}, 34\penalty0
  (3):\penalty0 1301--1340, 2013.
\newblock \doi{10.1137/120874540}.

\end{thebibliography}
\end{document}